\documentclass[french, 10pt, a4paper, oneside,bibliography=totocnumbered]{scrartcl}
\usepackage[T1]{fontenc} 
\usepackage[latin1]{inputenc} 
\usepackage{mathptmx} 
\usepackage{lmodern} 
\usepackage{amsthm} 
\usepackage{amsmath} 
\usepackage{amsfonts} 
\usepackage{amssymb} 
\usepackage{mathrsfs}
\usepackage[left=2.90cm, right=2.90cm, top=2.90cm, bottom=2.90cm]{geometry} 
\usepackage[ngerman,german,english]{babel} 
\usepackage{paralist}
\usepackage[colorlinks=true]{hyperref}
\hypersetup{linktocpage}
\usepackage{etoolbox}

\setkomafont{sectionentry}{\normalsize}
\RedeclareSectionCommand[tocbeforeskip=1pt]{section}

\setkomafont{section}{\normalfont\Large\textbf}
\setkomafont{subsection}{\normalfont\large\textbf}
\setkomafont{paragraph}{\normalfont\textbf}

\renewenvironment{abstract}
{\begin{center}
		\textbf{Abstract}
	\end{center}
	\list{}{ 
		\setlength{\leftmargin}{0.05\textwidth}
		\setlength{\rightmargin}{\leftmargin}
	}
	\item\relax} 
{\endlist}

\newenvironment{keywords}
{\begin{trivlist}\item[]{\bfseries Keywords.}}
	{\end{trivlist}}

\newenvironment{subclass}
{\begin{trivlist}\item[]{\bfseries Mathematical Subject Classification:}}
	{\end{trivlist}}

\newenvironment{acknowledgements}
{\begin{trivlist}\item[]{\bfseries Acknowledgements:}}
	{\end{trivlist}}

\theoremstyle{plain} 
\newtheorem{thm}{Theorem}[section] 

\newtheorem{lem}[thm]{Lemma}
\newtheorem{pro}[thm]{Proposition}

\theoremstyle{definition}
\newtheorem{defi}[thm]{Definition}
\newtheorem{rem}[thm]{Remark}

\newcommand{\N}{\mathbb{N}}

\newcommand{\R}{\mathbb{R}}

\newcommand{\norm}[2][]{\left\|#2\right\|_{#1}}

\newcommand{\skp}[2]{\left\langle #1,#2 \right\rangle}
\newcommand{\dualbra}[2]{\left\langle #1,#2 \right\rangle}

\newcommand{\Lpnorm}[2][]{\ifthenelse{\equal{#1}{}}{\norm{#2}_{L^p}}{\norm{#2}_{L^p(#1)}}}
\newcommand{\Hknorm}[2][]{\ifthenelse{\equal{#1}{}}{\norm{#2}_{H^k}}{\norm{#2}_{H^k(#1)}}}

\newcommand{\setdef}[2]{\left\lbrace #1 \ : \ #2 \right\rbrace}
\newcommand{\set}[1]{\left\lbrace #1 \right\rbrace}

\newcommand{\trace}{\text{\normalfont tr\,}}

\renewcommand{\div}{\text{\normalfont div}}

\newcommand{\QV}[2][]{\ifthenelse{\equal{#1}{}}{\langle #1 \rangle}{\langle #1,#2 \rangle}}

\newcommand{\ind}{\mathbb{I}}

\newcommand{\F}{\mathscr{F}}

\newcommand{\PrM}{\mathscr{P}}
\newcommand{\Lin}{\mathscr{L}}

\makeatletter
\newcounter{author}
\renewcommand*\author[1]{%
	\stepcounter{author}%
	\ifnum\c@author=1
	\gdef\@author{#1}%
	\else
	\xdef\@author{\unexpanded\expandafter{\@author\and#1}}%
	\fi
	\csgdef{author@\the\c@author}{#1}}
\newcommand*\email[1]{%
	\csgdef{email@\the\c@author}{#1}}
\newcommand*\address[1]{%
	\csgdef{address@\the\c@author}{#1}}
\AtEndDocument{%
	\xdef\author@count{\the\c@author}%
	\c@author=1
	\print@authors}
\newcommand*\print@authors{%
	\ifnum\c@author>\author@count
	\else
	\print@author{\the\c@author}%
	\advance\c@author by 1
	\expandafter\print@authors
	\fi}
\newcommand*\print@author[1]{%
	\par\medskip
	\begin{tabular}{@{}l@{}}%
		\textsc{\csuse{author@#1}}\\
		\csuse{address@#1}\\
		\textit{E-mail address}: \csuse{email@#1}
\end{tabular}}
\makeatother

\title{\textrm{\textbf{\Large Self-similar profiles for homoenergetic solutions of the Boltzmann equation for non-cutoff Maxwell molecules}}}

\author{\large Bernhard Kepka}
\address{University of Bonn, Institute for Applied Mathematics \\ Endenicher Allee 60 \\ D-53115 Bonn \\ GERMANY}
\email{kepka@iam.uni-bonn.de}

\date{\normalsize \today}

\begin{document}

\maketitle

\begin{abstract}
	We consider a modified Boltzmann equation which contains, together with the collision operator, an additional drift term that is characterized by a matrix $ A $. Furthermore, we consider a Maxwell gas, where the collision kernel has an angular singularity. Such an equation is used in the study of homoenergetic solutions to the Boltzmann equation. Our goal is to prove that, under smallness assumptions on the drift term, the longtime asymptotics is given by self-similar solutions. We work in the framework of measure-valued solutions with finite moments of order $ p>2 $ and show existence, uniqueness and stability of these self-similar solutions for sufficiently small $ A $. Furthermore, we prove that they have finite moments of arbitrary order if $ A $ is small enough. In addition, the singular collision operator allows to prove smoothness of these self-similar solutions. Finally, we study the asymptotics of particular homoenergetic solutions. This extends previous results from the cutoff case to non-cutoff Maxwell gases.
	\begin{subclass}
		35Q20, 82C40, 35C06
	\end{subclass}
	
	\begin{acknowledgements}
		The author thanks Juan J.L. Vel{\'a}zquez for the suggestion of the problem and helpful discussions. The author has been supported by the Deutsche Forschungsgemeinschaft (DFG, German Research Foundation) through the collaborative research centre \textit{The mathematics of emerging effects} (CRC 1060, Project-ID 211504053) and the Bonn International Graduate School of Mathematics at the Hausdorff Center for Mathematics (EXC 2047/1, Project-ID 390685813).
	\end{acknowledgements}
\end{abstract}

\tableofcontents

\section{Introduction}
The inhomogeneous Boltzmann equation is given by 
\begin{align}\label{eq:inhomBoltzmannEq}
	\partial_t f + v\cdot \nabla_x f = Q(f,f),
\end{align}
where $ f=f(t,x,v):[0,\infty)\times\R^3\times\R^3\rightarrow[0,\infty) $ is the one-particle distribution of a dilute gas in whole space. In this paper we will restrict ourselves to the physically most relevant case of three dimensions, although our study can be extended to dimensions $ d\geq 3 $ without any additional difficulties.

On the right-hand side we have Boltzmann's collision kernel
\begin{align*}
	Q(f,f) = \int_{\R^3}\int_{S^2}B(|v-v_*|,n\cdot \sigma)(f'_*f'-f_*f)d\sigma dv_*,
\end{align*}
where $ n=(v-v_*)/|v-v_*| $ and $ f'_* = f(v'_*) $, $ f' = f(v') $, $ f_*=f(v_*) $, with the pre-collisional velocities $ (v,v_*) $ resp. post-collisional velocities $ (v',v'_*) $. One parameterization of the post-collisional velocities is given by  the $ \sigma $-representation, i.e. for $ \sigma \in S^2 $
\begin{align}\label{eq:SigmaRep}
	\begin{split}
		v' = \dfrac{v+v_*}{2}+\dfrac{|v-v_*|}{2}\sigma, \qquad
		v'_* = \dfrac{v+v_*}{2}-\dfrac{|v-v_*|}{2}\sigma.
	\end{split}
\end{align}
Recall that the collision operator satisfies
\begin{align*}
	\int_{\R^3} Q(f,f) \varphi(v)dv=0, \quad \varphi(v)=1,\,v_1,\, v_2,\, v_3, \, |v|^2,
\end{align*}
which corresponds to the conservation of mass, momentum and energy. For an introduction into the physical and mathematical theory of the Boltzmann equation \eqref{eq:inhomBoltzmannEq} see for instance \cite{Cercignani1988BoltzmannEqApplication,Villani2002Review}.

The collision kernel is given by $ B(|v-v_*|,n\cdot \sigma) $ and it can be obtained from an analysis of the binary collisions of the gas molecules. For instance, power-law potentials $ 1/r^{q-1} $ with $ q>2 $ lead to (see e.g. \cite[Sec. II.5]{Cercignani1988BoltzmannEqApplication})
\begin{align}\label{eq:CollisinKernel}
	B(|v-v_*|,n\cdot \sigma)=|v-v_*|^{\gamma}\, b(n\cdot \sigma), \quad \gamma=(q-5)/(q-1),
\end{align}
where $ b:[-1,1)\rightarrow [0,\infty) $ has a non-integrable singularity of the form
\begin{align}\label{eq:angularSing}
	\sin\theta \, b(\cos\theta) \sim \theta^{-1-2/(q-1)}, \quad \text{as } \theta\to 0,
\end{align}
where $ \cos\theta = n\cdot \sigma $, with $ \theta $ being the deviation angle. It is customary to classify the collision kernels according to their homogeneity $ \gamma $ with respect to relative velocities $ |v-v_*| $. There are three cases: \textit{hard potentials} $ (\gamma>0) $, \textit{Maxwell molecules} $ (\gamma=0) $ and \textit{soft potentials} ($ \gamma<0 $).
In this paper we will consider the case of Maxwell molecules, hence $ B $ does not depend on $ |v-v_*| $, cf. \eqref{eq:CollisinKernel}. This corresponds to $ q=5 $ for power-law interactions.

Collision kernels with an angular singularity of the form \eqref{eq:angularSing} are called non-cutoff kernels. When $ \gamma=0 $, one refers to non-cutoff or true Maxwell molecules. This singularity reflects the fact that for power-law interactions the average number of \textit{grazing collisions}, i.e. collisions with $ v\approx v_* $, diverges. In kinetic theory the Boltzmann equation \eqref{eq:inhomBoltzmannEq} has then often been studied assuming that the collision kernel $ B $ is integrable in the angular variable, since the mathematical analysis is usually simpler. This assumption on the kernel is usually termed \textit{Grad's cutoff assumption}. The analysis of solutions to \eqref{eq:inhomBoltzmannEq} under such a cutoff assumption often allows to obtain some insight, even if such kernels might not be related to any physical interaction potential. Here, we consider \textit{non-cutoff} Maxwell molecules ($ \gamma=0 $).

In this paper, we analyze a particular class of solutions to \eqref{eq:inhomBoltzmannEq} namely the so-called \textit{homoenergetic solutions}, which have been studied in particular in \cite{BobylevVelazq2020SelfsimilarAsymp,JamesVelazq2019SelfsimilarProfilesHomoenerg} in the case of cutoff Maxwell molecules. We will show that the results obtain in their papers can be proved for non-cutoff Maxwell molecules, in particular for the kernel \eqref{eq:CollisinKernel} for interaction potentials $ 1/r^4 $. In this case, we have
\begin{align*}
	\sin\theta \, b(\cos\theta) \sim \theta^{-3/2}, \quad \text{as } \theta\to 0.
\end{align*}

\subsection{Homoenergetic solutions and existing results}
Our study concerns solutions to \eqref{eq:inhomBoltzmannEq} of the form
\begin{align}\label{eq:ansatzHomoenergSol}
	f(t,x,v)=g(t,v-L(t)x), \quad w=v-L(t)x,
\end{align} 
for $ L(t)\in\R^{3\times 3} $ and a function $ g=g(t,w):[0,\infty)\times\R^3\rightarrow[0,\infty) $ to be determined. One can check that solutions to \eqref{eq:inhomBoltzmannEq} of the form \eqref{eq:ansatzHomoenergSol} for large classes of functions $ g $ exist if and only if $ g $ and $ L $ satisfy
\begin{align}\label{eq:homenergBE}
	\begin{split}
		\partial_t g - L(t)w\cdot\nabla_w g &= Q(g,g), \\
		\dfrac{d}{dt}L(t) + L(t)^2 &= 0.
	\end{split}
\end{align}
The second equation allows the reduction to the variable $ w $. In particular, the collision operator acts on $ g $ through the variable $ w $. The second equation can be solved explicitly $ L(t)=L(0)(I+tL(0))^{-1} $. Note that the inverse matrix might not be defined for all times, although this situation will not be considered here.

Solutions to \eqref{eq:homenergBE} are called \textit{homoenergetic solutions} and were introduced by Truesdell \cite{Truesdell1956PressureFluxII} and Galkin 
\cite{Galkin1958ClassOfSolution}. They studied their properties via moment equations in the case of Maxwell molecules. As is known since the work by Truesdell and Muncaster \cite{TruesdellMuncasterFundamentalsMaxwell}, in the case of Maxwell molecules it is possible to write a closed systems of ordinary differential equations for the moments up to any arbitrary order. This allows to derive properties about the solution to \eqref{eq:homenergBE}. In particular, this approach has been applied in \cite{Galkin1958ClassOfSolution,Galkin1964OneDimUnsteadySol,Galkin1966ExactSol,Truesdell1956PressureFluxII}. More recently, this method has also been used in \cite{GarzoSantos2003KineticTheory} (and references therein) in order to obtain information on homoenergetic solutions to the Boltzmann equation, as well as other kinetic models like BGK. The case of mixtures of gases has been studied there as well. The well-posedness of \eqref{eq:homenergBE} for a large class of initial data, was proved by Cercignani \cite{Cercignani1989ExistenceHomoenergetic}. Furthermore, the shear flow of a granular material for Maxwell molecules was studied in \cite{Cercignani2001ShearFlowGranularMaterial,CercignaniBoltzmannApproachShearFlow}.

A systematic analysis of the large time behavior of solutions to \eqref{eq:homenergBE} for kernels with arbitrary homogeneities has been undertaken in \cite{BobylevVelazq2020SelfsimilarAsymp,VelazqJames2019LongAsympHomoen,JamesVelazq2019SelfsimilarProfilesHomoenerg,JamesVelazq2019LongtimeAsymptoticsHypDom}. In \cite{JamesVelazq2019SelfsimilarProfilesHomoenerg} the existence of a class of self-similar solutions in the case of (cutoff) Maxwell molecules has been proved. The uniqueness and stability of these self-similar solutions has been proved in \cite{BobylevVelazq2020SelfsimilarAsymp} and the regularity has been obtained in \cite{Liu2020BoltzmannUniformShearFlow}. Homoenergetic solutions for the two-dimensional Boltzmann equation with hard sphere interactions, as well as for a class of Fokker-Planck equations have been studied in \cite{Matthies2019}.

It is worth mentioning that homoenergetic solutions to \eqref{eq:inhomBoltzmannEq} can be interpreted in a wider framework introduced in \cite{DayalJamesNonequilibrium,DayalJames2012DesignViscometers}. There the authors introduced a formulation of the molecular dynamics of many interacting particle systems in the presence of symmetries. In particular, if the particles of the system of molecules of a gas interact by means of binary collisions one obtains the functional form \eqref{eq:ansatzHomoenergSol} for the particle distribution.

In this paper, we will make an extensive use of the Fourier transform method, which was introduced by Bobylev \cite{Boby1976FourierBoltzmannEq,Boby1988TheoryNonlinearBoltzmann} in order to study the homogeneous Boltzmann equation for Maxwell molecules. This method has also been used in \cite{BobylevVelazq2020SelfsimilarAsymp} to study the self-similar asymptotics of homoenergetic solutions for (cutoff) Maxwell molecules.

The main contribution of this paper is to see how to adapt the techniques in \cite{BobylevVelazq2020SelfsimilarAsymp} and well established methods for the non-cutoff Boltzmann equation to extend the results in \cite{BobylevVelazq2020SelfsimilarAsymp,JamesVelazq2019SelfsimilarProfilesHomoenerg}  to the case of non-cutoff Maxwell molecules. The main difficulty is the singular behavior of the collision kernel \eqref{eq:angularSing}.

We remark that self-similar solutions for the homogeneous Boltzmann equation were considered for elastic resp. inelastic collisions with infinite resp. finite energy in \cite{BobylevCercignani2002ExactEternalSol,BobylevCercignani2002SelfsimilarSol,BobylevCercignani2002SelfSimilarAsymptotics} for Maxwell molecules and in \cite{BobylevCercignani2002SelfsimilarSolBENonMaxwell} for non-Maxwellian molecules. They proved stability of self-similar solutions in \cite{BobylevCercignaniToscani2003ProofAsymptPropSelfsimilarSol} for inelastic collisions (with cutoff) and in \cite{BobylevCercignani2002SelfsimilarSol} for elastic collisions (without cutoff), see also \cite{BobylevCercignani2002SelfSimilarAsymptotics}. Furthermore, self-similar solutions were analyzed in a general framework of Maxwell models in \cite{BobylevCercignani2009SelfsimilarAsympt}. The case of true Maxwell molecules was also discussed in \cite{CannoneKarch2010InfintieEnergySol,CannoneKarch2013SelfsimilarSolHomogBE}. In particular, they proved smoothness based on a regularity result of the homogeneous Boltzmann equation for measure-valued solutions \cite{MorimotoYang2015SmoothingHomogBoltzmannEqMeasure}. Roughly speaking, this is a consequence of the fact that the singular collision operator behaves like the fractional Laplacian $ -(-\Delta)^{s} $, see e.g. \cite{AlexandreDesvillettesVillaniWennberg2000EntropyDissip}.

\subsection{Overview and main results}
\paragraph{Notation.} We will denote by $ \PrM(\R^3) $ the set of Borel probability measures on $ \R^3 $ and by $ \PrM_p(\R^3)\subset \PrM(\R^3) $ the set of those which have finite moments of order $ p $, i.e. $ \mu\in\PrM_p $ if
\begin{align*}
	\norm[p]{\mu} = \int_{\R^3}|v|^p\mu(dv)<\infty.
\end{align*}
The action of $ \mu\in\PrM $ on a test function $ \psi $ via integration will be abbreviated as $ \dualbra{\psi}{\mu} $. The Fourier transform or characteristic function of a probability measure $ \mu\in\PrM $ is defined by
\begin{align*}
	\varphi(k)= \F[\mu](k) = \int_{\R^3}e^{-ik\cdot x}d\mu(x).
\end{align*}
We will denote by $ \F_p $ the set of all characteristic functions of probability measures $ \mu\in\PrM_{p} $.

Concerning test functions, we will write $ \psi\in C^k $ for $ k $-times continuously differentiable functions and $ \psi\in C^k_b $ if the following norm is finite
\begin{align*}
	\norm[C^k]{\psi} := \sum_{0\leq \ell \leq k} \norm[\infty]{D^\ell \psi}<\infty.
\end{align*}
Here $ D^\ell \psi $ denotes the $ \ell $-th order differential. We also write $ C=C^0 $ for the case $ k=0 $ of continuous functions.

Furthermore, we use the notation $ \left\langle k \right\rangle := \sqrt{1+|k|^2} $ and denote the space of functions $ h:\R^3\to \R $ such that $ \left\langle k \right\rangle^m h(k)\in L^2(\R^3) $ by $ L^2_m(\R^3) $. 

For matrices $ A\in\R^{3\times 3} $ we use the matrix norm $ \norm{A}=\sum_{ij}|A_{ij}| $. Finally, we abbreviate $ a\wedge b= \min(a,b) $ for $ a,b\in\R $ and $ \ind_B $ is the indicator function for some set $ B $.

\paragraph{Assumption on the kernel.} We will consider non-cutoff Maxwell molecules, i.e. the collision kernel $ B $ does not depend on $ |v-v_*| $ and has the form $ B = b(n\cdot \sigma)=b(\cos\theta) $. The function $ b:[-1,1)\rightarrow[0,\infty) $ is measurable, locally bounded and has the angular singularity
\begin{align}\label{eq:AssumptCorssSectionNonCutoff}
	\sin\theta b(\cos\theta)\theta^{1+2s}\to K_b>0, \quad \text{as } \theta\to 0
\end{align}
for some $ s\in(0,1) $ and $ K_b>0 $. This implies
\begin{align}\label{eq:AssumptCrossSection}
	\Lambda =\int_0^{\pi} \sin\theta \, b(\cos\theta) \, \theta^2 d\theta <\infty.
\end{align}
In particular, this covers inverse power-law interactions with $ q=5 $, cf. \eqref{eq:CollisinKernel} and \eqref{eq:angularSing}.

\paragraph{Main result.} In our study we consider the following modified Boltzmann equation (and small perturbations of it), which is a variant of equation \eqref{eq:homenergBE},
\begin{align}\label{eq:GenHomoengBE}
	\partial_t f = \div(Av \, f) + Q(f,f), \quad f(0,\cdot)=f_0(\cdot).
\end{align}
In contrast to the previous equations, $ A\in\R^{3\times 3} $ is always a time-independent matrix. However, the study of solutions to \eqref{eq:homenergBE} can be reduced to this situation using a change of variables and perturbation arguments, see Section \ref{sec:Application}. We work with weak solutions with finite energy (i.e. second moments) that satisfy the following definition.

\begin{defi}\label{def:WeakForm}
	A family of probability measures $ (f_t)_{t\geq0}\subset \PrM_{p} $ with $ p\geq 2 $ is a weak solution to \eqref{eq:GenHomoengBE} if for all $ \psi\in C^2_b $ and all $ 0\leq t<\infty $ it holds
	\begin{align}\label{eq:WeakForm}
		\begin{split}
			\dualbra{\psi}{f_t}=&\dualbra{\psi}{f_0} - \int_{0}^{t}\dualbra{Av\cdot \nabla\psi}{f_r}dr 
			\\
			&+ \dfrac{1}{2}\int_0^t \int_{\R^3\times\R^3}\int_{S^2}b(n\cdot \sigma) \left\lbrace \psi'_*+\psi'-\psi_*-\psi \right\rbrace d\sigma f_r(dv)f_r(dv_*)dr.
		\end{split}
	\end{align}
	Here, we also assume that the integrands in the time integrals are measurable with respect to the time variable.
\end{defi}
Here we used $ \psi'_*=\psi(v'_*) $, etc. This formulation is motivated by multiplying \eqref{eq:GenHomoengBE} with $ \psi $, integrating in velocity space and time and applying the usual pre-post collisional change of variables $ (v,v_*)\leftrightarrow(v',v'_*) $ as well as $ v\leftrightarrow v_* $. See also e.g. \cite{JamesVelazq2019SelfsimilarProfilesHomoenerg,MouhotLu12MeasureSolBoltzmannEqPart1} concerning the above definition. For brevity we will sometimes denote the term involving the collision operator $ \dualbra{\psi}{Q(f_r,f_r)} $. Note that this is well-defined due to the moment assumption $ f_t\in \PrM_{p} $, $ p\geq2 $, in conjunction with the estimate
\begin{align}\label{eq:IntrodCancellation}
	\left|\int_{S^2} b(n\cdot \sigma) \left\lbrace \psi'_*+\psi'-\psi_*-\psi \right\rbrace d\sigma\right| \leq 2\pi\Lambda \,\left( \max_{|\xi|\leq \sqrt{|v|^2+|v_*|^2}}{|D^2\psi(\xi)|}\right) \, |v-v_*|^2.
\end{align}
This follows from a Taylor expansion up to order two and the use of spherical coordinates in the $ \sigma $-integral $ (\theta,\varphi)\in(0,\pi)\times(0,2\pi) $ such that $ \cos\theta = n\cdot\sigma $, $ n=(v-v_*)/|v-v_*| $ (see e.g. \cite{Villani1998NewClassWeakSol} or \cite{MouhotLu12MeasureSolBoltzmannEqPart1}).
Let us mention that one can always consider, without loss of generality, the case of vanishing momentum/mean $ \int_{\R^3} v f_0(dv)=0 $. To get a solution $ F $ with initial mean $ U\in\R^3 $ from $ f_t $, one defines $ F(t,v)=f_t(v-e^{tA}U) $ interpreted as a push-forward. However, as we will see, solutions with initial condition different from a Dirac measure are smooth for positive times due to the regularizing effect of the angular singularity.

Let us also define the following Fourier-based metric on probability measures.
\begin{defi}\label{def:ProbMetric}
	For two probability measures $ \mu,\,\nu\in\PrM_{p} $ with finite moments of order $ p\geq2 $ we define a distance using the Fourier transforms $ \varphi=\F[\mu],\, \psi=\F[\nu] $ via
	\begin{align*}
		d_2(\mu,\nu):=\sup_k\dfrac{|\varphi(k)-\psi(k)|}{|k|^2}.
	\end{align*}
\end{defi}
Note that $ d_2(\mu,\nu)<\infty $ is finite if $ \mu,\nu $ have equal first moments. With this let us state our main result.

\begin{thm}\label{thm:MainTheorem}
	Consider the equation \eqref{eq:GenHomoengBE}. Let $ 2< p \leq 4 $. There is a constant $ \varepsilon_0=\varepsilon_0(p,b)>0 $ such that if $ \norm{A}\leq \varepsilon_0 $, the following holds.
	\begin{enumerate}[(i)]
		\item There is $ \bar{\beta}=\bar{\beta}(A) $ and $ f_{st}\in\PrM_{p} $ so that \eqref{eq:GenHomoengBE} has a self-similar solution
		\begin{align}\label{eq:SelfSimilarSol}
			f(v,t)=e^{-3\bar{\beta}t}f_{st}\left( \dfrac{v-e^{-tA }U}{e^{\bar{\beta}t}} \right), \quad U\in\R^3,
		\end{align}
		where $ f_{st} $ has moments
		\begin{align*}
			\int_{\R^3}vf_{st}(dv)=0, \quad \int_{\R^3}v_iv_jf_{st}(dv)=K\bar{N}_{ij}.
		\end{align*}
		Here, $ K\geq0 $ and $ \bar{N}=\bar{N}(A)\in\R^{3\times 3} $ is a uniquely given positive definite, symmetric matrix with $ \sum_{ij}(\bar{N}_{ij})^2=1 $. For $ K=0 $, we have $ f_{st}=\delta_0 $, a Dirac measure in zero. 
		
		Furthermore, when $ K>0 $ the self-similar solutions are smooth
		\begin{align*}
			f(t,\cdot)\in L^1(\R^3)\cap \bigcap_{k\in\N} H^k(\R^3).
		\end{align*}
		
		\item Let $ (f_t)_t\subset \PrM_{p} $ be a weak solution to \eqref{eq:GenHomoengBE} with initial condition $ f_0\in\PrM_{p} $ and
		\begin{align*}
			U=\int_{\R^3}vf_0(dv).
		\end{align*}
		Then there is $ \alpha=\alpha(f_0)\in\R $, $ C=C(f_0,p)>0 $, $ \theta=\theta(\varepsilon_0)>0 $ such that the rescaled function
		\begin{align*}
			\tilde{f}(t,v):=e^{3\bar{\beta}t}f\left( e^{\bar{\beta}t}v+e^{-A t}U,t\right)
		\end{align*}
		satisfies
		\begin{align*}
			d_2\left(\tilde{f}(t,\cdot) , f_{st}(t,\cdot)  \right) \leq Ce^{- \theta t},
		\end{align*}
		where $ f_{st} $ is given in $ (i) $ with second moments $ \alpha^2\bar{N} $, $ K=\alpha^2 $. In particular, the self-similar solution in (i) is unique for given $ K\geq0 $.
		
		\item In addition, for all $ M\in\N $, $ M\geq 3 $ there is $ \varepsilon_M\leq \varepsilon_0 $ such that the self-similar solution from (i) has finite moments of order $ M $ if $ \norm{A}\leq \varepsilon_M $.
	\end{enumerate}
\end{thm}

\begin{rem}
	Note that $ f_{st} $ in (i) solves (in weak sense)
	\begin{align}\label{eq:GenStationaryBE}
		\div_v((A+\beta) v \cdot f_{st}) + Q(f_{st},f_{st})=0, \quad \beta=\bar{\beta}(A).
	\end{align}
	Furthermore, $ \bar{N} $ is a stationary solution to the second order moment equations ($ \bar{b}\in\R $ depending only on the collision kernel, see Lemma \ref{lem:MomentEq})
	\begin{align*}
		2\bar{\beta} \bar{N} - A\bar{N}-(A \bar{N})^\top -2\bar{b}\left( \bar{N}-\dfrac{\trace(\bar{N})}{3}I\right) = 0.
	\end{align*}
	As we will see, $ \bar{\beta}=\bar{\beta}(A) $ is chosen such that $ 2\bar{\beta} $ is the simple eigenvalue with largest real part (which is real). The corresponding eigenvector is given by $ \bar{N}=\bar{N}(A) $.
	
	The uniqueness result in $ (ii) $ can now be formulated in a more precise way: within the class of probability measures $ \PrM_{p} $, $ p>2 $, there is a unique solution $ f_{st}\in\PrM_{p} $ to the stationary equation \eqref{eq:GenStationaryBE} with $ \beta=\bar{\beta}(A) $ having moments 
	\begin{align*}
		\int_{\R^3}vf_{st}(dv)=0, \quad \int_{\R^3}v_iv_jf_{st}(dv)=\bar{N}_{ij}.
	\end{align*}
	It follows that $ f_{st}(K^{-1/2}v)K^{-3/2} $ has second moments $ K\bar{N}_{ij} $, $ K>0 $, and hence is the respective self-similar profile in $ (i) $. For $ K=0 $ this is a Dirac in zero.
\end{rem}

\begin{rem}
	The above theorem is similar to the results in \cite{BobylevVelazq2020SelfsimilarAsymp,JamesVelazq2019SelfsimilarProfilesHomoenerg}, where \textit{cutoff} Maxwell molecules have been considered. A comparison with Theorem \ref{thm:MainTheorem}, which covers the \textit{non-cutoff} case, shows that all results hold true under the same assumptions. Here, the smoothness statement in (i) is a consequence of the regularizing effect of the non-cutoff collision kernel, in contrast to the cutoff case \cite{Liu2020BoltzmannUniformShearFlow}, where this has been obtained in a smooth, perturbative framework close to a Maxwellian.
\end{rem}

\begin{rem}
	Regarding part $ (iii) $ in Theorem \ref{thm:MainTheorem} it might be that for small but fixed $ A\neq 0 $ the self-similar solutions do not have finite moments of arbitrary order, but that they have power-law tails. For shear flow this is suggested by numerical experiments, see \cite{GarzoSantos2003KineticTheory}.
	
	Let us also mention that the smallness of $ \norm{A} $ is crucial for our perturbation arguments. The precise behavior of solutions to \eqref{eq:GenHomoengBE} for large values of $ A $ remains open.
\end{rem}
The paper is organized in the following way. In Section \ref{sec:WellPosReg} we discuss the well-posedness theory of equation \eqref{eq:GenHomoengBE} and in Section \ref{sec:ExUniStabSelfSimilar} the proof of Theorem \ref{thm:MainTheorem}. Finally, in Section \ref{sec:Application} we come back to homoenergetic solution in the case of simple and planar shear and discuss the self-similar asymptotics.

\section{Well-posedness of the modified Boltzmann equation}
\label{sec:WellPosReg}

The aim of this section is to prove the following result, which summarizes the well-posedness theory of equation \eqref{eq:GenHomoengBE}, needed in our study.

\begin{pro}\label{pro:IntExistUniqReg}
	Under our general assumptions, the following statements hold.
	\begin{enumerate}[(i)]
		\item For all $ f_0\in\PrM_{p} $, $ p\geq2 $, there is a weak measure-valued solution $ (f_t)_t\subset\PrM_{p} $ to \eqref{eq:GenHomoengBE}. In addition, every weak solution has the property $ t\mapsto\dualbra{\psi}{f_t}\in C^1([0,\infty);\R) $ for all test functions $ \psi\in C^2 $ with $ \norm[\infty]{D^2\psi}<\infty $.
		
		\item For two weak solutions $ (f_t)_t, \, (g_t)_t\subset\PrM_{p} $ to \eqref{eq:GenHomoengBE}, $ p\geq2 $, such that $ f_0,\, g_0 $ have equal first moments, it holds
		\begin{align}\label{eq:UniquenssMetric}
			d_2(f_t,g_t)\leq e^{2\norm{A}t}d_2(f_0,g_0).
		\end{align}
		In particular, solutions are unique.
		
		\item If the initial datum $ f_0\in\PrM_{p} $, $ p\geq2 $, is not a Dirac measure,  the solution is smooth, i.e.
		\begin{align*}
			f(t,\cdot) \in L^1(\R^3)\cap \bigcap_{k\in\N} H^k(\R^3)
		\end{align*}
		for all $ t>0 $.
	\end{enumerate}
\end{pro}
\begin{rem}
	The setting of measure-valued solutions was also used in \cite{JamesVelazq2019SelfsimilarProfilesHomoenerg} for homoenergetic solutions. Measure-valued solutions to the homogeneous Boltzmann equation ($ A=0 $ in \eqref{eq:GenHomoengBE}) were considered in e.g. \cite{MouhotLu12MeasureSolBoltzmannEqPart1,MorimotoWang2016MeasureValuedSol} for both hard and soft potentials with homogeneity $ \gamma\geq -2 $. In \cite{MorimotoWang2016MeasureValuedSol} solutions with infinite energies are studied as well, see also \cite{CannoneKarch2010InfintieEnergySol,Morimoto2012RemarkCanoneKarchSol} for the case of Maxwell molecules.
	
	The metric in Definition \ref{def:ProbMetric} is also termed Toscani metric and appeared first in \cite{GabettaToscaniWennberg1995MetricsProbDistTrendEquili} for the study of convergence to equilibrium of the homogeneous Boltzmann equation with true Maxwell molecules. Furthermore, it was used to prove uniqueness of respective solutions in \cite{VillaniToscani1999}, by showing that solutions are contractive w.r.t. $ d_2 $. Inequality \eqref{eq:UniquenssMetric} is the extension of this Lipschitzianity to homoenergetic solutions. Further variants with a time-dependent matrix $ A $ are possible. The Toscani metric was also used in \cite{CannoneKarch2010InfintieEnergySol,MorimotoWang2016MeasureValuedSol} to construct solutions to the homogeneous Boltzmann equation with infinite energies. 
	
	The study of the homogeneous Boltzmann equation via the Fourier transform (for Maxwell molecules) goes back to Bobylev \cite{Boby1976FourierBoltzmannEq,Boby1988TheoryNonlinearBoltzmann} and since then was used often in the literature. See e.g. \cite{PulvirentiToscani1996TheoryBotlzEqMaxwellMolecules} for a study of the well-posedness theory of the homogeneous Boltzmann equation for true Maxwell molecules.
	
	Let us also mention that in the regime of the homogeneous Boltzmann equation uniqueness is in general false for $ \gamma\neq0 $, even for cutoff kernels. In fact, one can construct solutions with increasing energy \cite{LuWennberg2002SolutionIncEnergy,Wennberg1999ExampleNonuniq}. Such solutions do not exist for Maxwell molecules \cite{VillaniToscani1999}.
\end{rem}

The proof of Proposition \ref{pro:IntExistUniqReg} is based on well-posedness results in the cutoff case. Moment estimates as well as compactness properties for probability measures are used to obtain existence. The inequality \eqref{eq:UniquenssMetric} is derived using the Fourier transform and similar arguments as in \cite{VillaniToscani1999}. Finally, adapting the proof for the homogeneous Boltzmann equation in \cite{MorimotoYang2015SmoothingHomogBoltzmannEqMeasure} implies the smoothness result.

For our approximation arguments let us introduce an arbitrary cutoff sequence $ b_n : [-1,1)\rightarrow [0,\infty) $, $ b_n\not\equiv0 $, with $ b_n\nearrow b $, $ \norm[\infty]{b_n}<\infty $, e.g. $ b_n:=b\wedge n $ and denote the corresponding collision operators by $ Q_n $. The sequence of cutoff problems then read
\begin{align}\label{eq:CutoffGenHomoengBE}
	\partial_t f^n = \div(Av \, f^n) + Q^n(f^n,f^n), \quad f^n(0,\cdot)=f_0(\cdot).
\end{align}
Furthermore, let $ \Lambda_n $ be the corresponding constant as defined in \eqref{eq:AssumptCrossSection} with $ b_n $ replacing $ b $. Hence, by our choice of the cutoff $ \Lambda_n\leq \Lambda $. 

We will first prove part (ii) and for this reason and our later analysis summarize a view results of the modified Boltzmann equation \eqref{eq:GenHomoengBE} in Fourier space.

\subsection{The modified Boltzmann equation in Fourier space}
\label{subsec:FourierBE}
We reformulate the problem \eqref{eq:GenHomoengBE} respectively \eqref{eq:CutoffGenHomoengBE} in Fourier space. Consider a weak solution $ (f_t)_t\subset \PrM_{p} $, $ p\geq2 $ and its Fourier transform $ \varphi_t(k) = \F[f_t](k) $.

Recall the definition of the Fourier-based metric from Definition \ref{def:ProbMetric}. Since we will mostly work with characteristic functions, we will also use the notation $ d_2(\varphi,\psi)=d_2(\mu,\nu) $ for $ \varphi,\psi\in\F_p $. 

For a fixed $ k\in\R^3 $, we use $ \psi(v)=e^{-ik\cdot v} $ as a test function in the weak formulation of \eqref{eq:GenHomoengBE} yielding
\begin{align}\label{eq:FTgenBE}
	\partial_t\varphi_t(k) + A^\top k\cdot \nabla\varphi_t(k) = \hat{Q}(\varphi_t,\varphi_t)(k).
\end{align}
Note that part (i) in Proposition \ref{pro:IntExistUniqReg} implies that $ t\mapsto\varphi_t(k) \in C^1 $. However, we will make sure not to use it until the proof is settled. The last term in \eqref{eq:FTgenBE} corresponds to the collision operator. Since we work with Maxwell molecules, it has the form (Bobylev's formula \cite{Boby1976FourierBoltzmannEq,Boby1988TheoryNonlinearBoltzmann})
\begin{align}\label{eq:FTCollisionOp}
	\hat{Q}(\varphi,\varphi)(k) = \int_{S^2}b(\hat{k}\cdot \sigma) \left\lbrace \varphi(k_+)\varphi(k_-)-\varphi(k)\varphi(0) \right\rbrace d\sigma,
\end{align}
where $ k_\pm = (k\pm |k|\sigma)/2 $, $ \hat{k}=k/|k| $. Let us write $ \hat{Q}_n $ for the Fourier representation of the collision operator corresponding to a cutoff sequence $ 0\leq b_n\nearrow b $. We will often consider a decomposition of it in a gain and loss term. Thus, let us define
\begin{align}\label{eq:FTGainLossTerm}
	\begin{split}
		\hat{Q}_n^+(\varphi,\varphi)(k) &= \int_{S^2}b_n(\hat{k}\cdot \sigma) \varphi(k_+)\varphi(k_-)d\sigma,\\
		\hat{Q}_n^-(\varphi,\varphi)(k) &= \int_{S^2}b_n(\hat{k}\cdot \sigma)\varphi(k)\varphi(0) d\sigma = S_n \varphi(k).
	\end{split}
\end{align}
In the last equation, we used $ \varphi(0)=1 $ for characteristic functions and the constant
\begin{align}\label{eq:IntegralConstant}
	S_n:=\int_{S^2}b_n(e\cdot \sigma)d\sigma, \quad e\in S^2.
\end{align}
Observe that the integral does not depend on $ e\in S^2 $ by rotational invariance. This integral measures the average number of collisions and, since $ b $ is singular, we have $ S_n\nearrow+\infty $ as $ n\to \infty $. In particular, when analyzing the cutoff equation it is convenient to introduce the semigroup notation. The solution $ \varphi $ to \eqref{eq:FTgenBE} can then be written in mild form
\begin{align*}
	\varphi_t = e^{-A^\top k\cdot\nabla t}\varphi_0+\int_0^t e^{-A^\top k\cdot\nabla (t-r)}\hat{Q}_n(\varphi_r,\varphi_r)dr,
\end{align*}
where $ e^{-A^\top k\cdot\nabla t}\varphi(k)=\varphi(e^{-A^\top t}k) $.

Finally, let us recall the following property of characteristic functions.
\begin{lem}\label{lem:CharacteristicFunction}
	Consider $ \mu\in\PrM_{p} $, $ p>0 $, then its characteristic function satisfies $ \varphi\in C_b^{\lfloor p \rfloor,p-\lfloor p \rfloor} $ if $ p\notin\N $ and $ \varphi\in C_b^{p} $ if $ p\in\N $. Furthermore, 
	\begin{align*}
		\sup_k |\varphi(k)|\leq \varphi(0)=1, \quad \overline{\varphi(k)}=\varphi(-k).
	\end{align*}
\end{lem}
Let us note that $ \hat{Q}(\varphi,\varphi) $ is well-defined, even in the case of singular kernels. In fact, using $ \psi_t(k)= e^{-ik\cdot v} $ as a test function in the weak formulation and inferring the estimate \eqref{eq:IntrodCancellation}, we have
\begin{align}\label{eq:FourierCancellation}
	|\hat{Q}(\varphi,\varphi)(k)| \leq C\Lambda\, \norm[C^2]{\varphi} \, |k|^2.
\end{align}
Here, the second moments of the probability measure are estimated in terms of the second derivatives of its Fourier transform $ \varphi $, cf. Lemma \ref{lem:CharacteristicFunction}. Alternatively, one can use a Taylor expansion of $ \varphi $ up to second order in \eqref{eq:FTCollisionOp}. With a sequence of cutoff  kernels $ b_n\nearrow b $, we get
\begin{align}\label{eq:FourierCancellation2}
	|\hat{Q}(\varphi,\varphi)-\hat{Q}_n(\varphi,\varphi)|(k)\leq C(\Lambda-\Lambda_n) \, \norm[C^2]{\varphi} \, |k|^2.
\end{align}

\subsubsection{Linearization and Lipschitz property of the gain term}
For the Fourier transform of the cutoff operator $ \hat{Q}_n $ we introduce the linearization of $ \hat{Q}^+_n $ defined by
\begin{align}\label{eq:Linearization}
	\Lin_n(\varphi)(k)=\int_{S^2}b_n(\hat{k}\cdot \sigma)(\varphi(k_+)+\varphi(k_-))d\sigma,
\end{align} 
where $ \varphi\in C_b $, say. The following lemma is an adaptation of a similar result in \cite[Sect. 5]{BobylevVelazq2020SelfsimilarAsymp} and includes also singular kernels.

\begin{lem}\label{lem:EigenfuncEigenvalLinOp}
	Let us define
	\begin{align}\label{eq:EigenvalLinOp}
		\begin{split}
			w_p(s):=&1- \left( \dfrac{1+s}{2}\right)^{p/2}-\left( \dfrac{1-s}{2}\right)^{p/2},
			\\
			\lambda_n(p):=&\int_{S^2}b_n(e\cdot\sigma)\, w_p(e\cdot\sigma)d\sigma, \quad \lambda(p):=\int_{S^2}b(e\cdot\sigma)\, w_p(e\cdot\sigma)d\sigma.
		\end{split}
	\end{align}
	Then, $ \lambda(p) $ is well-defined for $ p\geq2 $ and $ \lambda_n(p)\to\lambda(p) $. Furthermore, $ \lambda(p) $ is strictly increasing w.r.t. $ p\geq2 $. In particular, we have $ \lambda(p)>\lambda(2)=0 $ for $ p>2 $.
\end{lem}
\begin{rem}\label{rem:EigenvalLinOp}
	We remark that $ |k|^p $, $ p>0 $, can be interpreted as an eigenfunction of the operator $ (\Lin_n-S_nI) $ w.r.t. the eigenvalue $ -\lambda_n(p) $, since we have
	\begin{align*}
		(\Lin_n-S_nI)|k|^p = -\lambda_n(p)|k|^p.
	\end{align*}
\end{rem}
\begin{proof}[Proof of Lemma \ref{lem:EigenfuncEigenvalLinOp}] First of all, let us use spherical coordinates $ (\theta,\varphi)\in(0,\pi)\times(0,2\pi) $ such that $ e\cdot \sigma = \cos\theta $ to obtain
	\begin{align*}
		\lambda(p) = 2\pi\int_0^\pi b(\cos \theta)\sin\theta\, \left(1-\cos^{p}\left( \theta/2\right) -\sin^{p}\left( \theta/2\right)\right)\, d\theta \leq C_p \Lambda.
	\end{align*}
	for $ p\geq2 $. The same formula holds for $ \lambda_n(p) $ when replacing $ b $ by $ b_n $. Thus, we see that $ \lambda_n(p)\to \lambda(p)  $. Note that $ 1-\cos^{p}\left( \theta/2\right) -\sin^{p}\left( \theta/2\right) $ is strictly increasing w.r.t. $ p $. Since $ b\geq0 $ is strictly positive due to the singularity at $ \theta =0 $, we conclude that $ \lambda(p) $ is strictly increasing. Finally, we also have $ \lambda(2)=0 $.
\end{proof}

The following estimate will play a crucial role in our analysis. This result is an adaptation of \cite[Lemma 3.1]{BobylevVelazq2020SelfsimilarAsymp}, where we made the dependence on the constant $ S_n $ explicit. See also the proof in \cite[Theorem 5]{VillaniToscani1999}.
\begin{lem}\label{lem:PosCutoffEstimateP}
	Consider two characteristic functions $ \varphi,\psi\in\F_p $, $ p\geq 2 $, and a cutoff sequence $ b_n\nearrow b $. Then, we have
	\begin{align}\label{eq:LLipschitz}
		|\hat{Q}_n^+(\varphi,\varphi)-\hat{Q}_n^+(\psi,\psi)|(k)\leq \Lin_n(|\varphi-\psi| )(k).
	\end{align}
	Moreover, it holds
	\begin{align}\label{eq:LinOpEstimate}
		\Lin_n(|\varphi-\psi| )(k)\leq  S_nd_2(\varphi,\psi)|k|^2.
	\end{align}
\end{lem}
\begin{proof}
	The first inequality follows from
	\begin{align*}
		|\varphi(k_+)\varphi(k_-)-\psi(k_+)\psi(k_-)| \leq |\varphi(k_+)-\psi(k_+)|+|\varphi(k_-)-\psi(k_-)|.
	\end{align*}
	The second one is a consequence of
	\begin{align*}
		\Lin_n(|\varphi-\psi|) \leq d_2(\varphi,\psi)\Lin_n(|k|^2) = S_nd_2(\varphi,\psi)|k|^2,
	\end{align*}
	recalling Remark \ref{rem:EigenvalLinOp} for the last identity.
\end{proof}
\begin{rem}\label{rem:LinearOp}
	The inequality \eqref{eq:LLipschitz} was termed $ \Lin $-Lipschitzianity in \cite{BobylevCercignani2009SelfsimilarAsympt}. It allows to estimate the difference $ |\psi_t-\varphi_t| $ of two solutions $ \varphi, \psi $ to \eqref{eq:FTgenBE} by the solution to the linear equation
	\begin{align}\label{eq:FTLinearEquation}
		\partial_t u + A^\top k\cdot \nabla u = (\Lin_n-S_n)u.
	\end{align} 
	with initial condition $ u_0=|\psi_0-\varphi_0| $. See Lemma \ref{lem:ComparisionCutoff} for a precise statement.  Recall from \eqref{eq:FTGainLossTerm} that $ \hat{Q}^-_n(\varphi_t,\varphi_t)= S_n\varphi_t $ due to $ \varphi_t(0)=1 $. 
	
	The difficulty arising in the non-cutoff case is that the linear operator on the right-hand side in \eqref{eq:FTLinearEquation} is a priori not well defined as $ n\to\infty $. However, the term $ (\Lin_n-S_n)u $ still makes sense for $ n\to \infty $ when $ u $ satisfies $ u(0)=0 $ and $ u\in C^2_b $. We will make use of this when choosing $ u(k) = |k|^p $ for $ p\geq2 $. As a result we obtain estimates of the form $ |\psi_t-\varphi_t|(k)\leq C|k|^p  $, which can then be used to bound $ d_2(\varphi_t,\psi_t) $.
\end{rem}

\subsection{Uniqueness of weak solutions}
\label{subsec:Uniquenss}
We turn to the proof of part (ii) of Proposition \ref{pro:IntExistUniqReg}. The argument is similar to the proof of uniqueness of weak solutions to the homogeneous Boltzmann equation in \cite{VillaniToscani1999}.
\begin{proof}[Proof of Proposition \ref{pro:IntExistUniqReg}. (ii).] 
	We split the proof into two parts.
	
	\textit{Step 1.} Let $ (f_t)_t\subset \PrM_p $, $ p\geq2 $, be a weak solution. We start with a priori bounds of moments of order $ p\geq2 $. To this end, fix $ T>0 $. Using the test function $ \psi(v)=|v|^p $, $ p\geq2 $, and \eqref{eq:IntrodCancellation} yields
	\begin{align*}
		\norm[p]{f_t}\leq \norm[p]{f_0}+\int_{0}^tp\norm{A}\norm[p]{f_r}dr + C_p\Lambda \int_{0}^t \norm[p]{f_r}dr.
	\end{align*}
	Gronwall's inequality yields $ \norm[p]{f_t}\leq  Ce^{C(p,A)T}\norm[p]{f_0} $. To make this rigorous one has to use an approximation of $ \psi $. Since $ (f_t)_t\subset \PrM_p $, one can then pass to the limit.
	
	Let now $ \psi(v)=e^{-ik\cdot v} $ for some fixed $ k\in\R^3 $. The above estimate and \eqref{eq:IntrodCancellation} implies 
	\begin{align*}
		\sup_{0\leq t\leq T}\dualbra{\psi}{Q(f_t,f_t)}<\infty
	\end{align*}
	for all $ k\in\R^3 $. We conclude from the weak formulation \eqref{eq:WeakForm} that $ t\mapsto \varphi_t(k)=\F[f_t](k) $ is Lipschitz for fixed $ k\in \R^3 $ and by Lemma \ref{lem:CharacteristicFunction} twice continuously differentiable w.r.t. $ k\in\R^3 $. In particular, \eqref{eq:FTgenBE} holds for all $ k\in\R^3 $ and a.e. $ t\geq0 $.
	
	\textit{Step 2.} Let $ (f_t)_t $, $ (g_t)_t $ be two weak solutions and $ \varphi_t(k)=\F[f_t](k) $, $ \psi_t(k)=\F[g_t](k) $ be the corresponding Fourier transforms. We can assume w.l.o.g. that $ d_2(f_0,g_0)<\infty $, otherwise there is nothing to show in \eqref{eq:UniquenssMetric}. Hence, the first moments $ U\in\R^3 $ are equal initially and they remain the same for all times. Let us assume w.l.o.g. that they are zero (by removing them with a shift $ v\mapsto v+e^{tA}U $ as mentioned in the introduction). Using a Taylor approximation together with the regularity properties of $ \varphi_t,\psi_t \in C^2_b $ one can show that $ d_2(\varphi_t,\psi_t) $ is finite and bounded by the $ C^2 $-norms of $ \varphi_t,\, \psi_t $. Since the moments of order $ p\geq2 $ are all bounded on say $ [0,T] $ for fixed $ T>0 $, by Lemma \ref{lem:CharacteristicFunction} these norms are also uniformly bounded on $ [0,T] $. Furthermore, we obtain from \eqref{eq:FourierCancellation2}
	\begin{align}\label{eq:ProofUniquenessEstimate}
		R_n(t,k):=\dfrac{1}{|k|^2}|(\hat{Q}-\hat{Q}_n)(\varphi_t,\varphi_t) - (\hat{Q}-\hat{Q}_n)(\psi_t,\psi_t)|(k)\leq C(T)r_n,
	\end{align}
	where we estimated the two terms containing $ (\hat{Q}-\hat{Q}_n) $ separately. Here $ r_n $ is independent of $ k $ and $ r_n\to 0 $ as $ n\to\infty $.
	
	Turning to the equation, let us for simplicity write $ A $ instead of $ A^\top $ noting that $ \norm{A^\top}=\norm{A} $ for our norm. Our goal is to prove a differential inequality and apply Gronwall's lemma. For this evaluate the equations \eqref{eq:FTgenBE} satisfied by $ (\varphi_t)_t, (\psi_t)_t $ at the point $ e^{tA}k $ and take their difference leading to
	\begin{align*}
		\dfrac{d}{dt}\left[ (\varphi_t-\psi_t)(e^{tA}k) \right] = \hat{Q}(\varphi_t,\varphi_t)(e^{tA}k)-\hat{Q}(\psi_t,\psi_t)(e^{tA}k).
	\end{align*}
	We use a cutoff $ \hat{Q}=\hat{Q}_n+\hat{Q}-\hat{Q}_n $ and decompose $ \hat{Q}_n $ into its gain and loss part
	\begin{align*}
		\dfrac{d}{dt}\left[ (\varphi_t-\psi_t)(e^{tA}k) \right] &= 	\hat{Q}_n^+(\varphi_t,\varphi_t)(e^{tA}k)- \hat{Q}_n^+(\psi_t,\psi_t)(e^{tA}k) 
		\\
		&- S_n(\varphi_t-\psi_t)(e^{tA}k)+ R_n|e^{tA}k|^2,
	\end{align*}
	where $ R_n $ from \eqref{eq:ProofUniquenessEstimate} is bounded $ \sup_k|R_n(k)|\leq C(T)r_n $. We multiply this with $ e^{S_nt} $ and absorb the loss part of the cutoff operator in the left side yielding
	\begin{align*}
		\dfrac{d}{dt}\left[ e^{S_nt}(\varphi_t-\psi_t)(e^{tA}k) \right] = e^{S_nt} \left[  \hat{Q}_n^+(\varphi_t,\varphi_t)(e^{tA}k)-\hat{Q}_n^+(\psi_t,\psi_t)(e^{tA}k) \right]
		\\
		+ e^{S_nt}R_n|e^{tA}k|^2.
	\end{align*}
	Now we divide by $ |e^{tA}k|^2\neq 0 $, for $ k\neq 0 $, to get
	\begin{align*}
		\dfrac{d}{dt}\left[ \dfrac{e^{S_nt}(\varphi_t-\psi_t)(e^{tA}k)}{|e^{tA}k|^2}\right] = \dfrac{2\skp{Ae^{tA}k}{e^{tA}k}}{|e^{tA}k|^2} \dfrac{e^{S_nt}(\varphi_t-\psi_t)(e^{tA}k)}{|e^{tA}k|^2} \\ +\dfrac{e^{S_nt}}{|e^{tA}k|^2}\left[ \hat{Q}_n^+(\varphi_t,\varphi_t)(e^{tA}k)-\hat{Q}_n^+(\psi_t,\psi_t)(e^{tA}k) \right]  + e^{S_nt}R_n.
	\end{align*}
	We integrate this in time, take absolute values and the supremum over $ k $. We estimate term by term and use Lemma \ref{lem:PosCutoffEstimateP} to bound the term containing $ \hat{Q}^+_n $. Let us introduce for convenience $ h_t(k):=(\varphi-\psi)(t,e^{tA}k)/|e^{tA}k|^2 $. We obtain
	\begin{align*}
		e^{S_nt}\norm[\infty]{h_t} \leq \norm[\infty]{h_0}+ \int_0^t \left[ 2\norm{A}+S_n \right]e^{S_nr}\norm[\infty]{h_r}dr + \int_0^t  e^{S_nr}C(T)r_ndr.
	\end{align*}
	Note that $ \norm[\infty]{h_t}=d_2(\varphi_t,\psi_t) $. We apply Gronwall's Lemma yielding
	\begin{align*}
		e^{S_nt}\norm[\infty]{h_t}\leq  \norm[\infty]{h_0}e^{\left[ 2\norm{A}+S_n \right]t} + C(T)r_n \int_0^te^{S_nr} e^{\left[ 2\norm{A}+S_n \right](t-r)}dr.
	\end{align*}
	We divide by $ e^{S_nt} $ and let $ n\to \infty $ so that
	\begin{align*}
		\norm[\infty]{h_t}\leq \norm[\infty]{h_0}e^{ 2\norm{A} t}.
	\end{align*}
	Recalling $ \norm[\infty]{h_t}=d_2(\varphi_t,\psi_t) $ concludes the proof.
\end{proof}

\subsection{Existence of weak solutions}
Here, we prove part (i) of Proposition \ref{pro:IntExistUniqReg}. For the particular case $ p=2 $, we need to recall the following fact concerning the topology induced by the metric $ d_2 $, see e.g. \cite[Lemma 1, Lemma 2]{VillaniToscani1999}.
\begin{lem}\label{lem:TopologyMetric}
	Define $ D_e\subset \PrM_2 $ by
	\begin{align*}
		D_e=\setdef{f\in\PrM_2}{\int vf(dv)=0, \quad \int|v|^2f(dv)=e}, \quad e\geq0.
	\end{align*}
	Consider $ f^n, f\in \PrM_2 $ for $ n\in\N $. Then, the following statements are equivalent:
	\begin{enumerate}[(i)]
		\item $ f_n,f\in D_e $ and $ f^n\rightharpoonup f $ weakly, i.e. $ \dualbra{\psi}{f^n}\to \dualbra{\psi}{f} $ as $ n\to \infty $ for all $ \psi \in C_b $;
		\item $ \lim_{K\to \infty} \sup_n \int_{|v|\geq K}|v|^2 f_n(dv) =0 $;
		\item $ d_2(f_n,f)\to 0 $ as $ n\to \infty $. 
	\end{enumerate}
\end{lem}

\begin{proof}[Proof of Proposition \ref{pro:IntExistUniqReg}. (i).] 
	We split this into several steps. In Step 1-3 we will cover the case $ p>2 $ and conclude in Step 4 with the case $ p=2 $.
	
	\textit{Step 1.} First of all, for all $ f_0\in\PrM_{p} $, $ p>2 $ one can prove the existence of a unique weak solution $ (f^n_t)_t\in C([0,\infty);\PrM_{p}) $ of the corresponding cutoff equation, i.e. with collision kernel $ Q_n $. The idea is to decompose $ Q_n=Q_n^+-Q_n^-  $ into its gain and loss part. With this one can use the semigroup generated by $ \div_v(Av \cdot )-Q_n^-(\cdot,\cdot) $ to treat the gain term as a perturbation via Duhamel's formula. Application of Banach's fixed point theorem proves existence and uniqueness of a mild solution, which is a weak solution. For more details see \cite[Sec. 4.1]{JamesVelazq2019SelfsimilarProfilesHomoenerg}. Let us note that due to conservation of mass we have a family of probability measures $ (f^n_t)_t\subset \PrM_{p} $.
	
	\textit{Step 2.} To get a solution to the non-cutoff equation on $ [0,T] $ we use a weak compactness argument. An a priori bound of the moments of order $ p>2 $ used before in Step 1 in the proof of uniqueness yields $ \norm[p]{f_t^n}\leq  Ce^{C(p,A)T}\norm[p]{f_0} $. From this, in conjunction with the weak formulation \eqref{eq:WeakForm}, we have the following continuity property independent of $ n\in\N $: for any test function $ \psi\in C^2 $ with $ \norm[C^2]{D^2\psi}<\infty $ and $ 0\leq s<t\leq T $
	\begin{align}\label{eq:equicont}
		\left| \dualbra{\psi}{f^n_t}-\dualbra{\psi}{f^n_s} \right| \leq (t-s)C\left(T,\norm[\infty]{D^2\psi},A,\Lambda\right)\norm[2]{f_0}.
	\end{align}
	
	Since for all $ t\in[0,T] $ the $ p $-moments of $ f_t^n $ are bounded uniformly in $ n\in\N $, the sequence is tight. From this compactness property and the above equicontinuity we can conclude that there is a weakly converging subsequence $ f_t^{n_k}\rightharpoonup f_t $ for all $ t\in[0,T] $. Furthermore, by the weak equicontinuity \eqref{eq:equicont} the mapping $ t\mapsto\dualbra{\psi}{f_t} $ is continuous and $ \norm[p]{f_t} $ is bounded on $ [0,T] $ due to our estimates.
	
	Finally, one has to pass to the limit in the weak formulation. First one considers test functions $ \psi\in C^2_b $, i.e. $ \norm[C^2]{\psi}<\infty $, and arguments similar to the one in \cite[Sec. 4]{MouhotLu12MeasureSolBoltzmannEqPart1}. More precise, the mapping 
	\begin{align*}
		L_{b_k}[\psi] (v,v_*)=\int_{S^2} b_k(n\cdot\sigma) \left\lbrace \psi(v')+\psi(v'_*)-\psi(v_*)-\psi(v) \right\rbrace\, d\sigma
	\end{align*}
	is continuous for each $ k\in\N $, bounded by $ C\Lambda\norm[\infty]{D^2\psi}\,|v-v_*|^2 $ and
	\begin{align*}
		\sup_{|v|+|v_*|\leq R} |L_{b_k}[\psi]-L_{b}[\psi]|\to 0
	\end{align*}
	as $ k\to \infty $ for all $ R>0 $, see \cite[Prop. 2.1]{MouhotLu12MeasureSolBoltzmannEqPart1}. In fact, $ L_{b_k}[\psi] $ has these properties even for $ \psi\in C^2 $ with $ \norm[\infty]{D^2\psi}<\infty $. Using the fact that $ f^n_t\otimes f^n_t \to f_t\otimes f_t $ and the bounds of moments of order $ p> 2 $ to treat velocities $ |v|+|v_*|\geq R $, we obtain
	\begin{align*}
		\dualbra{\psi}{Q_n(f^n_t,f_t^n)}\to \dualbra{\psi}{Q(f_t,f_t)}, \quad n\to \infty.
	\end{align*} 
	Using the dominated convergence theorem for the time integral in the weak formulation \eqref{eq:WeakForm}, we can pass to the limit.
	
	\textit{Step 3.} Finally, we prove that every weak solution $ (f_t)_t\subset \PrM_{p} $, $ p>2 $, has the property $ t\mapsto \dualbra{\psi}{f_t}\in C^1 $ for $ \psi\in C^2 $ with $ \norm[\infty]{D^2\psi}<\infty $. First of all, our a priori bounds in Step 2 yield $ \norm[p]{f_t} \leq C(T,A)\norm[p]{f_0} $ for all $ t\in[0,T] $. The weak formulation and \eqref{eq:IntrodCancellation} shows that $ \dualbra{\psi}{f_t} $ is Lipschitz. By approximation it follows that $ \dualbra{\psi}{f_t}  $ is continuous for all $ \psi\in C_b $, i.e. $ t\mapsto f_t $ is weakly-$ * $ continuous. Using the moment bounds of order $ p>2 $, this is also true in the case that $ \psi\ $ is merely continuous and has quadratic growth, i.e. $ |\psi(v)|\leq C(1+|v|^2) $ for some $ C\geq0 $.
	
	To prove that $ t\mapsto \dualbra{\psi}{f_t}\in C^1 $, it suffices to show that the integrand in the weak formulation is continuous. For a test function $ \psi \in C^2 $ with $ \norm[C^2]{D^2\psi}<\infty $ as was noticed above $ L_b[\psi] $ is continuous and has at most quadratic growth. Now, $ f_t\otimes f_t $ has the same continuity properties as $ f_t $ and we can conclude that $ \dualbra{\psi}{Q(f_t,f_t)} $ is continuous. Then, also the drift term $ \dualbra{Av\cdot \nabla\psi}{f_t} $ is continuous, since $ Av\cdot \nabla\psi $ has quadratic growth due to $ |\nabla \psi(v)|\leq C(1+|v|) $, recalling $ \norm[C^2]{D^2\psi}<\infty $.
	
	\textit{Step 4.} Now we treat the case $ p=2 $ with the case $ p>2 $ by approximation. Let $ (f_0^n)_n\subset\PrM_p $, $ p>2 $, be an approximation of $ f_0 $ such that $ d_2(f^n_0,f_0)\to0 $, in particular the energy is the same. The respective solutions $ (f_t^n)_t\subset\PrM_p $ are a Cauchy sequence by \eqref{eq:UniquenssMetric} in part (ii) of Proposition \ref{pro:IntExistUniqReg}. Hence, for all $ t\geq0 $ we get $ d_2(f^n_t,f_t)\to 0 $ as $ n\to \infty $ for the limit $ (f_t)_t\subset\PrM_2 $. By Lemma \ref{lem:TopologyMetric} we can pass to the limit in the weak formulation as in the case $ p>2 $. Recall that $ L_b[\psi] $ is continuous with quadratic growth for a test function $ \psi\in C^2_b $, so that $ \dualbra{\psi}{Q(f_t^n,f_t^n)}\to \dualbra{\psi}{Q(f_t,f_t)} $ as $ n\to \infty $. Hence, $ (f_t)_t $ is a weak solution.
	
	Finally, we want to show that $ \dualbra{\psi}{f_t} $ is $ C^1 $ for $ \psi\in C^2 $ with $ \norm[\infty]{D^2\psi}<\infty $. As in the case $ p>2 $, it suffices to show that $ \dualbra{\psi}{f_t} $ is continuous in $ t\geq0 $ for $ \psi $ continuous with at most quadratic growth. Due to Lemma \ref{lem:TopologyMetric} (ii) this is equivalent to continuity with respect to $ d_2 $, i.e. $ d(f_t,f_s)\to 0 $ as $ |t-s|\to0 $.
	
	In order to prove this, let $ (f_t^n)_t\subset\PrM_{p} $, $ p>2 $, be an approximation as before. By \eqref{eq:UniquenssMetric} we have $ d_2(f_t,f^n_t)\to 0 $ as $ n\to \infty $ uniformly in $ t\in[0,T] $ for some fixed $ T<\infty $. Since the approximations $ (f_t^n)_t $ are continuous w.r.t. $ t\geq0 $, when testing with continuous functions with quadratic growth, we obtain from Lemma \ref{lem:TopologyMetric} that $ d_2(f_t^n,f_s^n)\to 0 $ as $ |t-s|\to 0 $. Hence, for any sequence $ t_k\to t $ in $ [0,T] $ we get (by an $ \varepsilon/3 $-argument)
	\begin{align*}
		\limsup_{k\to \infty} \, d_2(f_{t_k},f_t)\leq \limsup_{k\to\infty}\left\lbrace  d_2(f_{t_k},f^n_{t_k})+d_2(f_{t_k}^n,f_{t}^n)+d_2(f_t^n,f_t) \right\rbrace \leq \varepsilon
	\end{align*}
	for all $ \varepsilon>0 $ and all sufficiently large $ n\in\N $. Thus, we obtain the required continuity.
\end{proof}

\subsection{Regularity of weak solutions}
We finally prove the regularity result in Proposition \ref{pro:IntExistUniqReg} (iii). We follow \cite{MorimotoYang2015SmoothingHomogBoltzmannEqMeasure}, which covers the homogeneous Boltzmann equation, i.e. $ A=0 $. We consider the Fourier transform $ (\psi_t)_t $ of a weak solution $ (f_t)_t\subset \PrM_2 $, hence
\begin{align}\label{eq:RegularityFTgenBE}
	\partial_t\psi_t(\xi) =- A^\top \xi\cdot \nabla\psi_t(\xi) + \hat{Q}(\psi_t,\psi_t)(\xi).
\end{align}
With the continuity property of $ t\mapsto f_t $, stated in part (i) of Proposition \ref{pro:IntExistUniqReg}, one can show that $ (t,\xi)\mapsto \psi(t,\xi), \, \partial_t\psi(t,\xi) $ is continuous. Let us first state a coercivity estimate analogous to the one in \cite[Lemma 1.4]{MorimotoYang2015SmoothingHomogBoltzmannEqMeasure}. As in the original work, the non-cutoff assumption \eqref{eq:AssumptCorssSectionNonCutoff} is essential.

\begin{lem}\label{lem:CoercivityEstimate}
	Consider $ (\psi_t)_t $ the Fourier transform of a solution $ (f_t)_t\subset \PrM_2 $ to \eqref{eq:GenHomoengBE}, i.e. $ (\psi_t)_t $ solves \eqref{eq:RegularityFTgenBE}. Assume that $ f_0\in\PrM_2 $ is not a Dirac measure. Then, there is $ T>0 $ and a constant $ C>0 $, both depending on $ f_0\in\PrM_2 $, such that for all $ h\in L^2_2(\R^3) $ and all $ t\in[0,T] $
	\begin{align}\label{eq:CoercivityEstimate}
		\begin{split}
			&t\int_{\R^3}\left\langle \xi \right\rangle^{2s}|h(\xi)|^2d\xi
			\\
			&\leq C\left\lbrace \int_{\R^3}\int_{S^2}b(\hat{\xi}\cdot \sigma)(1-|\psi(t,\xi_-)|)d\sigma |h(\xi)|^2 \, d\xi + \int_{\R^3}|h(\xi)|^2\,d\xi\right\rbrace.
		\end{split}
	\end{align}
	The constant $ s\in(0,1) $ is given in \eqref{eq:AssumptCorssSectionNonCutoff}.
\end{lem}
Let us not that the first integral on the right-hand side in \eqref{eq:CoercivityEstimate} is finite. To see this, one can use
\begin{align*}
	0\leq 1-|\psi(t,\xi_-)| \leq \dfrac{1-|\psi(t,\xi_-)|^2}{1+|\psi(t,\xi_-)|} \leq 1- \psi(t,\xi_-) \, \overline{\psi(t,\xi_-)}
\end{align*}
and now apply a Taylor expansion up to order two to get
\begin{align*}
	0\leq 1-|\psi(t,\xi_-)| \leq C\norm[C^2]{\psi_t}\left( |\xi_-|^2+|\xi_-|^4\right).
\end{align*}
Using spherical coordinates in the integral w.r.t $ \sigma $ with $ \cos\theta = \hat{\xi}\cdot \sigma $ and $ \xi_-=(\xi-|\xi|\sigma)/2 $ shows $ |\xi_-|=|\xi|\sin(\theta/2)\leq |\xi|\theta $. Hence, our assumption on the kernel , cf. \eqref{eq:AssumptCrossSection}, and $ h\in L^2_2 $ yield the finiteness.
\begin{proof}[Proof of Lemma \ref{lem:CoercivityEstimate}]
	The proof is very close to the original one for \cite[Lemma 1.4]{MorimotoYang2015SmoothingHomogBoltzmannEqMeasure}, since in most parts the equation \eqref{eq:RegularityFTgenBE} is not used but only the fact that $ \psi $ and $ \partial_t\psi $ are continuous.
	
	Let us give a brief sketch of the proof and the required adaptations. There are two cases. In the first case, one assumes that $ f_0 $ is not concentrated on a straight line and in the second case the opposite. Note that by our assumption $ f_0 $ is not a Dirac measure. The overall goal is to show that 
	\begin{align}\label{eq:ProofRegularityInequality}
		1-|\psi(t,\xi_-)|\geq c_0t
	\end{align}
	for some $ c_0>0 $ and $ \xi\in \R^3 $, $ \sigma\in S^2 $ belonging to certain regions and $ t\in[0,T] $, $ T>0 $. Let us choose spherical coordinates with $ \cos\theta = \hat{\xi}\cdot \sigma $ so that $ |\xi_-|= |\xi|\sin(\theta/2) $. The term $ \left\langle \xi \right\rangle^{2s} $ on the left integral in \eqref{eq:CoercivityEstimate} will be bounded using the singular behavior of the collision kernel when $ \theta\to 0 $. Note that at the same time inequality \eqref{eq:ProofRegularityInequality} fails when $ \xi_-=0 $, since $ \psi(t,0)=1 $. Nevertheless, since $ |\xi_-|=|\xi|\sin(\theta/2) $, one can find regions for $ \xi\in \R^3 $, $ \sigma\in S^2 $ such that $ |\xi| $ is arbitrarily large, $ \theta $ close to zero and $ |\xi_-| $ bounded away from zero. Evaluating the $ \sigma $-integral, in conjunction with the singularity $ \sin\theta \, b(\cos\theta) \sim \theta^{-1-2s} $, allows to estimate large values of $ |\xi| $ in \eqref{eq:CoercivityEstimate}. The part of the integral, where $ |\xi| $ is of order one, is bounded by the second term on the right-hand side in \eqref{eq:CoercivityEstimate}.
	
	In the first case, where $ f_0 $ is not concentrated on a line, it is proved in \cite{MorimotoYang2015SmoothingHomogBoltzmannEqMeasure} that there exist $ e\in S^2 $, $ \lambda>0 $ and $ \kappa_0>0 $, depending on $ f_0 $, such that
	\begin{align*}
		1-|\psi(0,\xi)|\geq \kappa_0>0
	\end{align*}
	for all $ \xi\in \R^3 $ with $ \xi\cdot e=0 $, $ |\xi|=\lambda $. By continuity one obtains
	\begin{align*}
		1-|\psi(t,\xi)|\geq \kappa_0/2
	\end{align*}
	for all $ t\in[0,T] $ with $ T>0 $ sufficiently small and $ \xi \in \mathcal{C}_{\delta,\varepsilon} $ with
	\begin{align*}
		\mathcal{C}_{\delta,\varepsilon} := \setdef{\eta\in\R^3}{\lambda-\delta\leq|\eta|\leq\lambda+\delta, \; \left|\hat{\eta}\cdot e \right|\leq\varepsilon},
	\end{align*}
	$ \varepsilon, \delta>0 $ sufficiently small. In order to get \eqref{eq:ProofRegularityInequality} for $ \xi\in \R^3 $, $ |\xi| $ large, we need to find regions for $ \sigma \in S^2 $ such that $ \xi_-\in \mathcal{C}_{\delta,\varepsilon} $. A geometrical argument (see \cite{MorimotoYang2015SmoothingHomogBoltzmannEqMeasure}) shows that for all $ \xi\in \R^3 $, $ |\xi| \geq R $, $ R>0 $ large, but fixed, one can find a set $ B(\xi)\subset S^2 $ such that $ \sigma\in B(\xi) $ implies $ \xi_-\in \mathcal{C}_{\delta,\varepsilon} $. To describe the main idea, let us use spherical coordinates $ (\theta,\varphi)\in(0,\pi)\times(0,2\pi) $ with $ \cos\theta= \hat{\xi}\cdot \sigma $. One can see that $ |\xi_-|=|\xi|\sin(\theta/2) $ and hence, we need for $ \xi_-\in\mathcal{C}_{\delta,\varepsilon} $
	\begin{align*}
		2\arcsin\left( \dfrac{\lambda-\delta}{|\xi|} \right) \leq \theta \leq 2\arcsin\left( \dfrac{\lambda+\delta}{|\xi|} \right).
	\end{align*} 
	Furthermore, one can find an interval $ I_\xi $ for $ \varphi\in I_{\xi}\subset(0,2\pi) $, so that, with the above condition for $ \theta $, we get $ \xi_-\in \mathcal{C}_{\delta,\varepsilon} $. Then, we can estimate
	\begin{align}\label{eq:ProofRegularityCoercivity}
		\int_{B(\xi)}b(\hat{\xi}\cdot \sigma)(1-|\psi(t,\xi_-)|)d\sigma \geq C\kappa_0 \left\langle \xi \right\rangle^{2s}.
	\end{align}
	
	For the second case, when $ f_0 $ is supported on a straight line, we need to adapt the arguments in \cite{MorimotoYang2015SmoothingHomogBoltzmannEqMeasure}, since now the equation \eqref{eq:RegularityFTgenBE} is used. Here, we have in addition the drift term. We can assume w.l.o.g. that $ f_0 $ is supported on the line spanned by $ e_3=(0,0,1) $, i.e. it is of the form $ f_0(dv)=\delta_0(dv')f_{03}(dv_3) $, $ v'=(v_1,v_2) $. Hence, the Fourier transform has the form $ \psi_0(\xi)= \psi_{03}(\xi_3) $. Arguing as in \cite{MorimotoYang2015SmoothingHomogBoltzmannEqMeasure} allows to show that there is a $ \lambda>0 $, $ c_1>0 $, depending on $ f_0 $, such that for all $ \xi\in \R^3 $ with $ \xi\cdot e_3=0 $, $ |\xi|=\lambda $
	\begin{align*}
		2\text{Re} [\hat{Q}(\psi_0,\psi_0)\overline{\psi_0}](\xi) = 2\text{Re} \, \int_{S^2} b\left(  \dfrac{\xi}{|\xi|}\cdot \sigma \right) \left( \psi_0^+\, \psi_0^- \, \overline{\psi_0}-|\psi_0|^2 \right) \, d\sigma  \leq -c_1.
	\end{align*} 
	Now, we use equation \eqref{eq:RegularityFTgenBE} and write $ \tilde{\psi}(t,\xi) := \psi(t,e^{-A^\top t}\xi) $ which yields
	\begin{align*}
		\partial_t \left[ \left| \tilde{\psi}(t,\xi) \right|^2 \right]\mid_{t=0} = 2\text{Re}\left[ \hat{Q}(\psi_0,\psi_0)(\xi)\bar{\psi_0}\right] (\xi) \leq -c_1
	\end{align*}
	for all $ \xi\in \R^3 $ such that $ |\xi|=\lambda  $, $ \xi\cdot e_3=0 $. By continuity we can extend this to $ (t,\xi)\in [0,T]\times \Gamma $
	\begin{align*}
		\partial_t \left[ \left| \tilde{\psi}(t,\xi) \right|^2 \right]\leq -c_1/2,
	\end{align*}
	where $ T>0 $ is sufficiently small and 
	\begin{align*}
		\Gamma = \setdef{\xi \in\R^3}{\left| |\xi|-\lambda \right|\leq \delta, \; \left| \dfrac{\xi}{|\xi|}\cdot e_3  \right|\leq \varepsilon },
	\end{align*}
	$ \varepsilon, \delta>0 $ sufficiently small. This implies for all $ (t,\xi_-)\in [0,T]\times \Gamma $, using the mean value theorem,
	\begin{align*}
		1-\left| \psi \left( t,e^{-A^\top t}\xi_- \right) \right|= 1-|\tilde{\psi}(t,\xi_-)|\geq \dfrac{1-|\tilde{\psi}(t,\xi_-)|^2}{2}\geq \dfrac{1}{2}\left[ 1-|\tilde{\psi}(0,\xi_-)|^2+c_1t/2 \right] \geq \dfrac{c_1t}{4},
	\end{align*}
	by noting $ 1-|\tilde{\psi}(0,\xi_-)|^2= 1-|\psi_0(\xi_-)|^2\geq 0 $. We can formulate this in the following way: for all $ (t,\xi_-)\in [0,T]\times e^{-A^\top t}\Gamma $ we have
	\begin{align}\label{eq:LemProofCoercivityEst}
		1-|\psi(t,\xi_-)|\geq \dfrac{c_1t}{4},
	\end{align}
	where we defined
	\begin{align*}
		e^{-A^\top t}\Gamma = \setdef{\eta \in\R^3}{\eta = e^{-A^\top t}\xi, \; \xi \in\Gamma}.
	\end{align*}
	Since $ e^{-A^\top t}  $ is close to the identity matrix for small times, we can choose $ T, \delta, \varepsilon>0 $ smaller, if necessary, to get this for all $ (t,\xi_-)\in [0,T]\times \Gamma $. Again, one needs to find conditions for $ \xi, \sigma $ such that $ \xi_-\in \Gamma $, in the same manner as in the previous case, see \cite{MorimotoYang2015SmoothingHomogBoltzmannEqMeasure}. This yields for $ \xi\in \R^3 $ with $ |\xi|\geq R $ and $ \xi/|\xi| $ close to $ \pm e_3 $ a set $ B(\xi)\subset S^2 $. As for \eqref{eq:ProofRegularityCoercivity} we can estimate
	\begin{align*}
		\int_{B(\xi)}b(\hat{\xi}\cdot \sigma)(1-|\psi(t,\xi_-)|)d\sigma \geq Ct\left\langle \xi \right\rangle^{2s}.
	\end{align*}
	For other directions of $ \xi $, the original arguments in \cite{MorimotoYang2015SmoothingHomogBoltzmannEqMeasure} can be used without changes, since the equation \eqref{eq:RegularityFTgenBE} is not used but only the continuity of $ \psi $. 
\end{proof}

\begin{proof}[Proof of Proposition \ref{pro:IntExistUniqReg}. (iii).]
	We follow the arguments in \cite[Proof of Thm. 1.3]{MorimotoYang2015SmoothingHomogBoltzmannEqMeasure}. 
	
	\textit{Step 1.} In this step, we prove smoothness only for $ 0<t\leq T/2 $ with $ T>0 $ as in Lemma \ref{lem:CoercivityEstimate}. Define
	\begin{align*}
		M_\delta(t,\xi):=\left\langle \xi \right\rangle^{Nt^2-4}\left\langle \delta \xi \right\rangle^{-NT^2-4}, \quad N\in\N,
	\end{align*}
	where $ N $ is chosen large enough such that $ h= M_\delta\psi\in L^2_2 $ for $ t\leq T/2 $. We multiply the equation \eqref{eq:RegularityFTgenBE} with $ M_\delta^2 \overline{\psi} $, integrate over $ \R^3 $ w.r.t. $ \xi $ and take the real part. Note that $ M_\delta $ decays sufficiently fast at infinity for $ t\leq T/2 $. We then obtain the equation
	\begin{align}\label{eq:RegularityProofTestingEq}
		2\text{Re}\int_{\R^3}\left[ \partial_t\psi\, M_\delta^2 \, \overline{\psi}\right]\, d\xi = -2\text{Re}\int_{\R^3}A^\top \xi\cdot \nabla\psi \, M_\delta^2 \, \overline{\psi} \, d\xi + 2\text{Re} \int_{\R^3} \hat{Q}(\psi,\psi) \, M_\delta^2 \,\overline{\psi}\, d\xi
	\end{align}
	For the collision term we use the coercivity estimate \eqref{eq:CoercivityEstimate} with $ h=|M_\delta\psi|\in L^2_2 $ and arguments as in \cite{MorimotoYang2015SmoothingHomogBoltzmannEqMeasure}, which do not use equation \eqref{eq:RegularityFTgenBE}. We obtain for some constants $ C_1,C_2>0 $
	\begin{align*}
		2\text{Re} \int_{\R^3} \hat{Q}(\psi_t,\psi_t)(\xi) \, \left(  M_\delta^2 \,\overline{\psi} \right) (t,\xi)\, d\xi \leq C_1\int_{\R^3}|M_\delta(t,\xi)\, \psi(t,\xi)|^2d\xi - C_2t\int_{\R^3}\left\langle \xi \right\rangle^{2s} |M_\delta(t,\xi)\, \psi(t,\xi)|^2d\xi.
	\end{align*}
	On the left-hand side of \eqref{eq:RegularityProofTestingEq} we use
	\begin{align*}
		2\text{Re}\left[ \partial_t\psi\, M_\delta^2 \, \overline{\psi}\right](t,\xi) = \partial_t\left[ |M_\delta\, \psi|^2\right] (t,\xi) - 4Nt\log\left\langle \xi \right\rangle \,|M_\delta\, \psi|^2(t,\xi).
	\end{align*}
	For the first term on the right in \eqref{eq:RegularityProofTestingEq} we compute with partial integration
	\begin{align*}
		&-\int_{\R^3}A^\top \xi\cdot \nabla\psi(t,\xi) \, M_\delta(t,\xi)^2 \, \overline{\psi}(t,\xi) \, d\xi = \int \trace(A)|M_\delta\psi|^2(t,\xi)d\xi \\
		&+ \int_{\R^3} 2A^\top \xi\cdot \nabla M_\delta M_\delta \, |\psi(t,\xi)|^2 \, d\xi + \int_{\R^3}A^\top \xi\cdot\nabla\overline{\psi}(t,\xi)  \, M_\delta(t,\xi)^2 \, \psi(t,\xi) \, d\xi.
	\end{align*}
	Note that the last term coincides up to the sign with the term on the left-hand side due to $ \overline{\psi}(t,\xi)=\psi(t,-\xi) $ and the change of variables $ \xi\mapsto-\xi $. Furthermore, $ A^\top \xi\cdot \nabla M_\delta(t,\xi) $ is bounded by $ C(T,A)M_\delta $ for some constant $ C(T,A)>0 $. So we can estimate
	\begin{align*}
		\left| 2\text{Re} \int_{\R^3}A^\top \xi\cdot\nabla\psi \, M^2_\delta \overline{\psi} \, d\xi \right| \leq C(T,A)\int_{\R^3}|M_\delta\psi|^2(t,\xi) \, d\xi.
	\end{align*}
	Putting all terms together in \eqref{eq:RegularityProofTestingEq} yields
	\begin{align}\label{eq:RegularityGronwall}
		\dfrac{d}{dt}\int_{\R^3}|M_\delta(t,\xi)\psi(t,\xi)|^2 \, d\xi &\leq C(T,A)\int_{\R^3}|M_\delta(t,\xi)\, \psi(t,\xi)|^2d\xi \\
		&+t\int_{\R^3}\left[ 4N\log\left\langle \xi \right\rangle-C_2\left\langle \xi \right\rangle^{2s} \right] |M_\delta(t,\xi)\, \psi(t,\xi)|^2d\xi.
	\end{align}
	Since $ \left\langle \xi \right\rangle^{2s}/\log\left\langle \xi \right\rangle \to \infty $ as $ |\xi|\to \infty $, the last term on the right can be absorbed in the first term. Using Gronwall's lemma and letting $ \delta\to 0 $ one obtains
	\begin{align*}
		\int_{\R^3}|\left\langle \xi \right\rangle^{Nt^2-4}\psi(t,\xi)|^2\, d\xi\leq C \int_{\R^3}|\left\langle \xi \right\rangle^{-4}\psi_0(\xi)|^2 \, d\xi.
	\end{align*}
	Since this holds for all $ N\in\N $, we have $ f(t,\cdot)\in\bigcap_{k\in\N} H^k(\R^3) $ for $ 0<t\leq T/2 $.
	
	\textit{Step 2.} In this step, we extend the regularity to times $ t\in [t_0,T'] $ with $ T'>T/2 $ arbitrary and some $ t_0\in (0,T/2) $ fixed.
	
	Since $ f_{t_0}\in L^2\cap L^1_2 $, with $ L^1_2(\R^3) $ the weighted $ L^1 $-space with weight $ (1+|v|^2) $, $ f_{t_0} $ has finite entropy, i.e.
	\begin{align*}
		H(f_{t_0}):=\int_{\R^3} f(t_0,v) \log f(t_0,v)\, dv <\infty.
	\end{align*}
	To see this, split the integral into the regions $ \{f\geq 1\} $, $ \{\mu\leq f\leq 1\} $ and $ \{f\leq \mu\} $. On $ \{f\geq 1\} $ one uses the $ L^2 $-bound and on $ \{\mu\leq f\leq 1\} $ the $ L^1_2 $-bound. For the last region $ \{f\leq \mu\} $ apply the elementary inequality $ x\log x \geq x-y-x\log y $, $ 0<x\leq y $, with $ x=f $, $ y=\mu $.
	
	We would like to propagate this bound on $ [t_0,T'] $. Testing formally the equation \eqref{eq:GenHomoengBE} with $ \log f(t,v) $ and using that $ \dualbra{Q(f,f)}{\log f}\leq 0 $ yields with a Gronwall argument
	\begin{align}\label{eq:EntropyBound}
		H(f_t)\leq H(f_{t_0})+C(T',A), \quad t\in [t_0,T'].
	\end{align}
	This is just formal, since we do not know yet that $ f_t $ is a function for times $ t>t_0 $. We can make it rigorous in the following way. We can construct a weak solution to \eqref{eq:GenHomoengBE} in $ L^1_2 $ with finite entropy with initial condition $ f_{t_0} $. Such a construction was done in \cite{Cercignani1989ExistenceHomoenergetic} for homoenergetic solutions with angular cutoff using classical arguments from \cite{Arkeryd72OnBoltzmannEq}. Here, we use weak solutions to compensate also for the angular singularity. The weak form is exactly the same as in Definition \ref{def:WeakForm} but now $ f_t $ are functions in $ L^1_2 $. Our existence proof can be used without much change. The only difference is the compactness argument in order to pass from an approximating sequence $ f^n $ to $ f $. For this one, uses the additional bound on the entropy \eqref{eq:EntropyBound}, which is uniform in $ n\in\N $. The weak $ L^1 $-compactness, following from the Dunford-Pettis theorem, ensures that the limit $ f $ is an $ L^1 $-function. See e.g. the discussion in \cite[Sec. 4]{Villani1998NewClassWeakSol} for the homogeneous Boltzmann equation without angular cutoff. Since we proved already uniqueness on the level of measure-valued solutions, the so constructed solution coincides with our measure-valued solution. As a consequence, $ f_t\in L^1_2 $, $ t\in [t_0,T'] $ and the bound \eqref{eq:EntropyBound} holds.
	
	We continue to follow the final arguments in \cite{MorimotoYang2015SmoothingHomogBoltzmannEqMeasure}. We know use Lemma 3 in \cite{AlexandreDesvillettesVillaniWennberg2000EntropyDissip}, which ensures that any function $ g\in L^1_2 $ with finite entropy $ H(g)<\infty $ satisfies 
	\begin{align*}
		1-|\psi(\xi) |\geq \kappa \min(1,|\xi|^2)
	\end{align*}
	with $ \psi = \F[g] $ the Fourier transform. Here, $ \kappa>0 $ only depends on upper bounds of $ \norm[L^1_2]{g} $ and $ H(g) $. As a consequence of the bound
	\begin{align*}
		\sup_{t\in[t_0,T']} \left[ \norm[L^1_2]{f_t} + H(f_t)\right] <\infty,
	\end{align*}
	we have 
	\begin{align*}
		1-|\psi(t,k) |\geq \kappa \min(1,|\xi|^2)
	\end{align*}
	with $ \kappa>0 $ independent of $ t\in[t_0,T'] $. This yields a coercivity estimate similar to \eqref{eq:CoercivityEstimate} without the condition of small times
	\begin{align*}
		\int_{\R^3}\left\langle \xi \right\rangle^{2s}|h(\xi)|^2d\xi
		\leq C\int_{\R^3}\int_{S^2}b(\hat{\xi}\cdot \sigma)(1-|\psi(t,\xi_-)|)d\sigma |h(\xi)|^2 \, d\xi, \quad t\in [t_0,T'].
	\end{align*}
	The arguments in Step 1 hence apply for all times $ t\in[t_0,T'] $ yielding instead of \eqref{eq:RegularityGronwall}
	\begin{align*}
		\dfrac{d}{dt}\int_{\R^3}|M_\delta(t,\xi)\psi(t,\xi)|^2 \, d\xi &\leq C(T,A)\int_{\R^3}|M_\delta(t,\xi)\, \psi(t,\xi)|^2d\xi \\
		&+\int_{\R^3}\left[ 4NT'\log\left\langle \xi \right\rangle-C_2\left\langle \xi \right\rangle^{2s} \right] |M_\delta(t,\xi)\, \psi(t,\xi)|^2d\xi.
	\end{align*} 
	Thus, as above we get $ f(t,\cdot)\in\bigcap_{k\in\N} H^k(\R^3) $ for $ t\in[t_0,T'] $.  Since $ T'>0 $ was arbitrary, this concludes the proof.
\end{proof}

\section{Self-similar solutions and self-similar asymptotics}
\label{sec:ExUniStabSelfSimilar}
In this section, we will give the proof of Theorem \ref{thm:MainTheorem}. As in our assumption of the main theorem, we consider always weak solutions $ (f_t)_t\subset \PrM_p $ with $ p>2 $. By Proposition \ref{pro:IntExistUniqReg} (iii) they are smooth for positive times when $ f_0 $ is not a Dirac measure. However, since the smoothness of solutions does not play any role in our arguments, we work within $ \PrM_p $.

Let us briefly summarize the strategy of the proof of Theorem \ref{thm:MainTheorem}. First, we study the linear ordinary differential equations satisfied by the second moments of a solution to \eqref{eq:GenHomoengBE}. We do this by analyzing for small $ A $ the spectrum of the involved operator, which acts on symmetric matrices, using a perturbation argument. This allows to show that there is a simple, real eigenvalue $ 2\bar{\beta}(A) $, which has largest real part among all eigenvalues, and its eigenvector $ \bar{N}(A)\in \R^{3\times 3} $ is symmetric, positive definite. With this we prove the existence of a self-similar solution $ f_{st} $, cf. Theorem \ref{thm:MainTheorem} (i), by an application of Povzner estimates and a fixed point argument. This strategy is reminiscent of the arguments in the cutoff case \cite{JamesVelazq2019SelfsimilarProfilesHomoenerg}.

For the stability result in Theorem \ref{thm:MainTheorem} (ii), we use the Fourier transform and as a key ingredient a comparison principle based on the linearized cutoff equation in Fourier space. With this, and a longtime analysis of the second order moment equations, we conclude the stability statement. Similar arguments in Fourier space have been used in \cite{BobylevVelazq2020SelfsimilarAsymp}.

Finally, the last statement, Theorem \ref{thm:MainTheorem} (iii), is a result of successive application of Povzner estimates. At this point, we need to choose $ \norm{A} $ small enough to deal with the drift term.

\subsection{Existence of self-similar solutions}
\label{subsec:Existence}
For the existence result in Theorem \ref{thm:MainTheorem} (i), let us recall the following version of the Povzner estimate due to Mischler and Wennberg \cite[Sect. 2]{MischlerWennberg1999HomogBE}. As was noticed e.g. in \cite[Appendix]{Villani1998NewClassWeakSol}, their calculation also works in the non-cutoff case.

\begin{lem}\label{lem:Povzner}
	Let $ \varphi(v)=|v|^{2+\delta} $ for $ \delta>0 $. Then we have the following decomposition
	\begin{align*}
		\int_{S^2}b(n\cdot \sigma )\left\lbrace \varphi'_*+\varphi'-\varphi_*-\varphi \right\rbrace d\sigma = G(v,v_*)-H(v,v_*)
	\end{align*}
	with $ G,H $ satisfying
	\begin{align}
		\begin{split}\label{eq:PovznerDecomposition}
			G(v,v_*) &\leq C\Lambda(|v||v_*|)^{1+\delta/2},
			\\
			H(v,v_*) &\geq c\Lambda(|v|^{2+\delta}+|v_*|^{2+\delta}) \left( 1-\ind_{\set{|v|/2< |v_*|< 2|v|}} \right).
		\end{split}
	\end{align}
	Hence, for any $ f\in\PrM_{p} $, with $ 2<p\leq 4 $, $ p=2+\delta $ it holds
	\begin{align}\label{eq:PovznerBound}
		\begin{split}
			\int_{\R^3\times\R^3}(G(v,v_*)-H(v,v_*))f(dv)f(dv_*) 
			\\
			\leq C'\Lambda\norm[2]{f}^2-c'\Lambda\norm[p]{f}.
		\end{split}
	\end{align}
\end{lem}
\begin{proof}
	The definition and estimates for $ G,H $ can be found in \cite[Sect. 2]{MischlerWennberg1999HomogBE}, see also \cite[Appendix]{Villani1998NewClassWeakSol}. To derive \eqref{eq:PovznerBound} note that the term involving $ G $ can be estimated via \eqref{eq:PovznerDecomposition} and $ \delta = p-2\leq 2 $, hence $ 1+\delta/2\leq 2 $. With
	\begin{align*}
		(|v|^{2+\delta}+|v_*|^{2+\delta})\ind_{\set{|v|/2< |v_*|< 2|v|}} \leq 8(|v||v_*|)^{1+\delta/2}
	\end{align*}
	the other term is bounded by
	\begin{align*}
		-2c\Lambda\norm[p]{f} + 4c\Lambda\norm[2]{f}^2.
	\end{align*}
\end{proof}
An important ingredient in our analysis will be the moment equations satisfied by solutions to~\eqref{eq:GenHomoengBE}.
\begin{lem}\label{lem:MomentEq}
	Consider a weak solution $ (f_t)_t\subset \PrM_{p} $, $ 2<p $, to \eqref{eq:GenHomoengBE}. Then, the second moments $ M_{jk}(t):=\dualbra{v_jv_k}{f_t} $ satisfy the equations
	\begin{align}\label{eq:MomentEq}
		\dfrac{dM_t}{dt} =- AM_t-(A M_t)^\top -2\bar{b}\left( M_t-\dfrac{\trace(M_t)}{3}I\right) =: \mathcal{A}(\bar{b},A)M_t
	\end{align}
	with the constant
	\begin{align}\label{eq:MomentEqConst}
		\bar{b} = \dfrac{3\pi}{4}\int_0^\pi b(\cos\theta) \sin^3\theta d\theta.
	\end{align}
	Here, the linear operator $ \mathcal{A}(\alpha,A): \R^{3\times 3}_
	{sym}\rightarrow \R^{3\times 3}_{sym} $ acts on symmetric $ 3\times 3 $ matrices and depends on $ \alpha\in\R $, $ A\in\R^{3\times 3} $.
\end{lem}
Note that \eqref{eq:MomentEqConst} is well-defined by our assumption \eqref{eq:AssumptCrossSection} on the kernel.
\begin{proof}[Proof of Lemma \ref{lem:MomentEq}]
	We can choose $ \varphi_{jk}(v)=v_jv_k $ in the weak formulation \eqref{eq:WeakForm} and recall that $ t\mapsto\dualbra{\varphi_{jk}}{F_t} $ is continuously differentiable. A calculation using spherical coordinates with north pole $ n=(v-v_*)/|v-v_*| $ yields the constant \eqref{eq:MomentEqConst}. For more details see e.g. \cite[Prop. 4.10]{JamesVelazq2019SelfsimilarProfilesHomoenerg} or \cite[Sect. 6]{BobylevVelazq2020SelfsimilarAsymp}.
\end{proof}
Our goal now is to prove the existence of steady states of the form
\begin{align*}
	f(v,t)=e^{-3\bar{\beta}t}f_{st}\left( \dfrac{v}{e^{\bar{\beta}t}} \right).
\end{align*}
Note that we can restrict to the case of zero momentum, as it was noticed in the introduction. Hence, we look for a solution $ f_{st} $ of
\begin{align}\label{eq:StatGenBE}
	\beta \div_v (vf_{st})+\div_v(Av\cdot f_{st}) + Q(f_{st},f_{st})=0
\end{align}
with $ \beta\in\R $, $ A\in\R^{3\times 3} $. For convenience, we will sometimes write $ A_\beta=A+\beta I = A+\beta $. Again, equation \eqref{eq:StatGenBE} has to be considered in weak form. More precisely, for all $ \psi\in C^2_b $ it holds
\begin{align}\label{eq:WeakFormSteadyState}
	\dualbra{(A+\beta)v\cdot\nabla\psi}{f_{st}} = \dualbra{\psi}{Q(f_{st},f_{st})}.
\end{align}
As in the cutoff case \cite{JamesVelazq2019SelfsimilarProfilesHomoenerg}, we study the corresponding second moment equations \eqref{eq:MomentEq} with matrix $ A+\beta $ replacing $ A $.

\begin{lem}\label{lem:SteadyStateMomenEq}
	Consider the linear operator $ \mathcal{A}(\bar{b},A) $ from Lemma \ref{lem:MomentEq} with parameters $ A\in\R^{3\times 3} $ and $ \bar{b} $ in \eqref{eq:MomentEqConst}. There is a sufficiently small constant $ \varepsilon_0=\varepsilon_0(b)>0 $ such that for all $ A\in\R^{3\time3} $ with $ \norm{A}\leq \varepsilon_0 $ the following holds. The eigenvalue $ 2\bar{\beta}>0 $, $ \bar{\beta}=\bar{\beta}(A,\bar{b}) $, with largest real part is unique and simple. One can uniquely choose a corresponding eigenvector $ \bar{N}=\bar{N}(A,\bar{b})\in\R^{3\times 3}_{sym} $ which is positive definite with $ \norm{N}=1 $. Furthermore, the nonzero eigenvalues of $ \mathcal{A}(\bar{b},A)-2\bar{\beta} $ have real part less than or equal to $ -\nu $, for some $ \nu>0 $. In addition, there is $ c_0>0 $ such that $ |\bar{\beta}(A,\bar{b})|\leq c_0\varepsilon_0 $.
\end{lem}
\begin{proof}
	This is a perturbative argument noting that $ \mathcal{A}(\bar{b},A):\R^{3\times 3}_{sym}\rightarrow \R^{3\times 3}_{sym} $ depends smoothly on $ A $. For $ A=0 $ there are the eigenvalues $ \beta = 0 $ and $ 2\beta=-2\bar{b} $ with a one-dimensional subspace of eigenvectors given by $ M=KI $, $ K\in\R $, respectively, a five-dimensional subspace of eigenvectors defined by $ \set{\trace(M)=0} $. The statement now follows by continuity results for eigenvalues when $ A $ is close to $ 0 $, which defines $ \varepsilon_0 $. We choose $ 2\bar{\beta}(A,\bar{b}) $ to be the eigenvalue close to $ 2\bar{\beta}(0,\bar{b})=0 $ and let $ \bar{N}(A,\bar{b})\in \R^{3\times 3}_{sym} $ be the corresponding normalized eigenvector close to $ I $. Hence, for sufficiently small $ \varepsilon_0 $ the matrix $ \bar{N}(A,\bar{b}) $ is positive-definite. Note that (for small $ \varepsilon_0 $) the eigenspace is one-dimensional and, since the coefficients in the system are real, the corresponding eigenvalue $ 2\bar{\beta}(A,\bar{b}) $ is real. The other eigenvalues are close to $ -2\bar{b}<0 $ and thus, $ 2\bar{\beta}(A,\bar{b}) $ has the largest real part among them. In particular, there is $ \nu>0 $ such that the real part of the nonzero eigenvalues $ \mathcal{A}(\bar{b},A)-2\bar{\beta} $ are less than or equal to $ -\nu $. Finally, the eigenvalues are continuous with respect to $ A $ and thus the estimate $ |\bar{\beta}|\leq c_0\varepsilon_0 $ follows from $ \bar{\beta}(0,\bar{b})=0 $.
\end{proof}
As a final ingredient, we give a moment estimate, which is uniform in time, that will be useful for our compactness argument (see also \cite[Prop. 4.20]{JamesVelazq2019SelfsimilarProfilesHomoenerg}).
\begin{lem}\label{lem:UnifMomentBound}
	Let $ (f_t)\in C([0,\infty);\PrM_{p}) $ be the unique weak solution to \eqref{eq:GenHomoengBE} with matrix $ A $ replaced by $ A+\bar{\beta} $. Assume that $ \norm{A}\leq \varepsilon_0 $ with $ \varepsilon_0>0 $ from Lemma \ref{lem:SteadyStateMomenEq} and that the initial condition $ f_0\in\PrM_{p} $, $ 2<p\leq 4 $, satisfies
	\begin{align}\label{eq:MomentsInital}
		\int_{\R^3}vf_0(dv)=0, \quad \int_{\R^3}v_iv_jf_0(dv)=K\bar{N}^n_{ij},
	\end{align}
	with $ K\geq0 $. Then, we have for all $ t\geq 0 $
	\begin{align}\label{eq:MomentsLater}
		\int_{\R^3}vf_t(dv)=0, \quad \int_{\R^3}v_iv_jf_t(dv)=K\bar{N}_{ij}.
	\end{align}
	Furthermore, by decreasing $ \varepsilon_0=\varepsilon_0(b,p)>0 $, if necessary, there is $ C_*=C_*(K) $ such that
	\begin{align*}
		\norm[p]{f_0}\leq C_* \implies \norm[p]{f_t}\leq C_*
	\end{align*}
	for all $ t\geq0 $.
\end{lem}

\begin{proof}
	As was mentioned already in the introduction, the first moment remains zero for all times. The statement about the second moments follows from Lemma \ref{lem:MomentEq} and the fact that $ \bar{N} $ are stationary solutions to the corresponding moment equations \eqref{eq:MomentEq}. For the final statement, we use the Povzner estimate from Lemma \ref{lem:Povzner}. Using the test function $ \psi(v)=|v|^p $ in the weak formulation for $ f $ and \eqref{eq:PovznerBound} we deduce
	\begin{align*}
		\dfrac{d}{dt}\norm[p]{f_t}&=\dfrac{d}{dt}\dualbra{|v|^p}{f_t}\leq p\norm{A+\bar{\beta}}\norm[p]{f_r} + C'\Lambda \norm[2]{f_r}^2 - c'\Lambda\norm[p]{f_r}
		\\
		&\leq \left[ p\varepsilon_0(1+c_0)-c'\Lambda\right] \norm[p]{f_r} + C'\Lambda K^2.
	\end{align*}
	For $ \varepsilon_0 $ sufficiently small we have $ \delta:=c'\Lambda-p\varepsilon_0(1+c_0)>0 $ and hence from a Gronwall type argument, in conjunction with $ \norm[p]{f_0}\leq C_* $,
	\begin{align*}
		\norm[p]{f_t} \leq C_*e^{-\delta t} + \dfrac{C'\Lambda K^2}{\delta} =C_* + \left( 1-e^{-\delta t}\right) \left( \dfrac{C'\Lambda K^2}{\delta}-C_*\right)
	\end{align*}
	We conclude by choosing $ C_*=C_*(K) $ sufficiently large.
\end{proof}
We are now able to prove our main existence result.
\begin{proof}[Proof of Theorem \ref{thm:MainTheorem}. (i)]
	We use similar arguments as in \cite[Sect. 4.3]{JamesVelazq2019SelfsimilarProfilesHomoenerg}. Let us define the set $ \mathscr{U}\subset\PrM_{p} $, $ 2<p\leq 4 $, consisting of measures $ f\in\PrM_{p} $ with
	\begin{align*}
		\int_{\R^3}vf(dv)=0, \quad \int_{\R^3}v_iv_jf(dv)=K\bar{N}_{ij}, \quad \norm[p]{f}\leq C_*.
	\end{align*}
	Here, $ \bar{N} $ is given in Lemma \ref{lem:SteadyStateMomenEq} and we assume that $ \norm{A}\leq \varepsilon_{0} $ as in Lemmas \ref{lem:SteadyStateMomenEq}, \ref{lem:UnifMomentBound}. Note that $ \mathscr{U} $ is a convex, compact subset of the space $ \mathscr{M}_{f}(\R^3) $ of signed Radon measures on $ \R^3 $ with finite total variation, equipped with the weak-$ * $ topology. With this topology $ \mathscr{M}_{f}(\R^3) $ is a locally convex space. Let us note that weak convergence within $ \mathscr{U} $ implies convergence w.r.t. the metric $ d_2 $ by Lemma~\ref{lem:TopologyMetric}.
	
	Let us define the nonlinear semigroup $ \mathscr{S}_t:\PrM_{p}\rightarrow \PrM_{p} $ mapping any $ f_0 $ to $ f_t $, where $ (f_t)_t $ is the unique solution to the equation \eqref{eq:GenHomoengBE} with matrix $ A+\bar{\beta} $ replacing $ A $ and initial condition $ f_0 $. By Lemma \ref{lem:UnifMomentBound} we have $ \mathscr{S}_t:\mathscr{U}\rightarrow \mathscr{U} $. Furthermore, $ f\mapsto \mathscr{S}_tf $ is continuous on $ \mathscr{U} $ for each $ t\geq0 $, as follows from \eqref{eq:UniquenssMetric}.
	
	We can now apply Schauder's fixed point theorem to the continuous self-mappings $ \mathscr{S}_{1/n}:\mathscr{U}\rightarrow \mathscr{U} $, which yields a fixed point $ f^n_{st} $. By compactness of $ \mathscr{U} $ we have for a subsequence $ f^{n_k}_{st}\to f_{st} $ as $ k\to \infty $. As a consequence of the semigroup property, it holds $ \mathscr{S}_{m/n_k}f^{n_k}_{st}=f^{n_k}_{st} $ for any $ k,m\in\N $. 
	
	Let now $ t\geq0 $ be arbitrary. We can find a sequence of integers $ m_k\in\N $ with $ m_k/n_k\to t $ as $ k\to \infty $ and write
	\begin{align*}
		f_{st}= \lim_{k\to\infty} f^{n_k}_{st} = \lim_{k\to\infty} \mathscr{S}_{m_k/n_k}f^{n_k}_{st} = \mathscr{S}_{t}f_{st}.
	\end{align*}
	To verify the last equality, we use \eqref{eq:UniquenssMetric} and estimate
	\begin{align}\label{eq:ProofFixedPointEstimate}
		\begin{split}
			d_2\left( \mathscr{S}_{m_k/n_k}f^{n_k}_{st},\mathscr{S}_{t}f_{st} \right) &\leq d_2\left( \mathscr{S}_{m_k/n_k}f^{n_k}_{st},\mathscr{S}_{m_k/n_k}f_{st}  \right) + d_2\left( 	\mathscr{S}_{m_k/n_k}f_{st},\mathscr{S}_{t}f_{st} \right) 
			\\
			&\leq e^{2(t+1)\norm{A+\bar{\beta}}} d_2\left( f^{n_k}_{st},f_{st} \right) + d_2\left( \mathscr{S}_{m_k/n_k}f_{st},\mathscr{S}_{t}f_{st} \right).
		\end{split}
	\end{align}
	The first term goes to zero, since $ f^{n_k}_{st}\to f_{st} $ in $ \mathscr{U} $. For the last term, recall from Proposition \ref{pro:IntExistUniqReg} (i) that $ t\mapsto\dualbra{\psi}{f_t} $ is continuous for all test functions $ \psi\in C^2 $ with $ \norm[\infty]{D^2\psi}<\infty $, where $ f_t:=\mathscr{S}_{t}f_{st} $. By an approximation this holds for all continuous functions with at most quadratic growth. Hence, we can conclude, by using again Lemma \ref{lem:TopologyMetric}, that the last term in \eqref{eq:ProofFixedPointEstimate} goes to zero.
	
	In total $ f_{st} $ is a stationary solution, i.e. it satisfies \eqref{eq:StatGenBE} with $ \beta $ replaced by $ \bar{\beta} $. We thus constructed the self-similar solution
	\begin{align*}
		f(v,t)=e^{-3\bar{\beta}t}f_{st}\left( v e^{-\bar{\beta}t}\right)
	\end{align*} 
	of \eqref{eq:GenHomoengBE}, which has zero mean and is smooth by Proposition \ref{pro:IntExistUniqReg} (iii). To obtain mean $ U\in\R^3 $ we use the change of variables $ v\mapsto v-e^{-tA }U $, as mentioned in the introduction. Finally, one can see that the Dirac measure $ f_{st}=\delta_0 $ is a solution to \eqref{eq:WeakFormSteadyState}, yielding a self-similar profile with $ K=0 $. This concludes the existence proof.
\end{proof}

\subsection{Uniqueness and stability of self-similar solutions}
For the proof of part (ii) in Theorem \ref{thm:MainTheorem}, we will again use the Fourier transform $ \varphi_t(k)=\F[f_t](k) $ for a solution $ (f_t)_t\subset\PrM_p $, $ p>2 $. A key ingredient will be a comparison principle between two solutions. This is based on the $ \Lin $-Lipschitzianity property of the gain term $ \hat{Q}^+_n $ from Lemma \ref{lem:PosCutoffEstimateP}. Recall that $ \hat{Q}^+_n $ is the gain term in Fourier space corresponding to a cutoff kernel $ b_n\nearrow b $, cf. \eqref{eq:FTGainLossTerm}. More precisely, we first consider the linearization of the cutoff equation given by
\begin{align}\label{eq:CutoffLinPDE}
	\partial_t\varphi + A^\top k\cdot\nabla\varphi= (\Lin_n-S_nI)(\varphi)(k), \quad \varphi(0,\cdot)=\varphi_0(\cdot).
\end{align}
Recall that $ \Lin_n $, $ S_n $ are defined in \eqref{eq:Linearization} and \eqref{eq:IntegralConstant}, respectively. One can prove that \eqref{eq:CutoffLinPDE} defines a semigroup $ \mathcal{P}^n_t:C_p\rightarrow C_p $, where
\begin{align*}
	C_p(\R^3):=\setdef{\varphi\in C(\R^3)}{\norm[C_p]{\varphi}:=\sup_k |\varphi(k)|/(1+|k|^p)<\infty}
\end{align*}
for all $ p\geq 2 $. This follows essentially from the fact that $ \Lin_n:C_p\rightarrow C_p $ is bounded, see \cite{BobylevVelazq2020SelfsimilarAsymp}.

In the following lemma, we show that the difference of two solutions to the cutoff problem \eqref{eq:CutoffGenHomoengBE} can be be estimated using the linearized problem \eqref{eq:FTLinearEquation}.
\begin{lem}\label{lem:ComparisionCutoff}
	Consider two solutions $ \varphi, \, \psi\in C([0,\infty);\F_p) $ to \eqref{eq:FTgenBE} with cutoff collision kernel $ b_n $. Assume that $ |\varphi_0-\psi_0|(k)\leq u_0(k) $, $ u_0\in C_p $. Then, we have for all $ k\in\R^3 $, $ t\geq0 $
	\begin{align*}
		|\varphi_t(k)-\psi_t(k)|\leq \mathcal{P}^n_t [u_0]  (k).
	\end{align*}
\end{lem}
\begin{proof} We sketch the arguments. A detailed proof can be found in \cite[Sect. 5]{BobylevVelazq2020SelfsimilarAsymp}.
	
	\textit{Step 1.} Since the semigroup $ \exp(-(S_n+A^\top k\cdot\nabla)t) $ is monotonicity preserving and $ \Lin_n $ is positivity preserving, it follows: if $ v_0,u_0\in C_p $ with $ 0\leq v_0\leq u_0 $, then 
	\begin{align*}
		0\leq \mathcal{P}^n_tv_0(\cdot)\leq \mathcal{P}^n_tu_0(\cdot).
	\end{align*}
	Note that
	\begin{align*}
		\exp(-(S_n+A^\top k\cdot\nabla)t) \varphi (k) = e^{-S_nt}\varphi(e^{-A^\top t}k).
	\end{align*}
	
	\textit{Step 2.} Let $ \varphi, \, \psi $ be two solutions to \eqref{eq:FTgenBE} then we can write in mild form
	\begin{align*}
		\varphi_t(k)-\psi_t(k) &= \varphi_0(k)-\psi_0(k) 
		\\
		&+\int_0^t e^{-(S_n+A^\top k\cdot\nabla)(t-r)}\left[ \hat{Q}^+_n(\varphi_r,\varphi_r)-\hat{Q}^+_n(\psi_r,\psi_r) \right](k)dr. 
	\end{align*}
	Set $ v_t(k):=\varphi_t(k)-\psi_t(k) $ and estimate using the $ \Lin $-Lipschitz property in Lemma \ref{lem:PosCutoffEstimateP}
	\begin{align*}
		|v_t(k)|\leq |v_0(k)|+\int_0^t e^{-(S_n+A^\top k \cdot \nabla)(t-r)} \Lin_n(|v_r|)(k)dr.
	\end{align*}
	A comparison principle for the linear equation implies $ |v_t(k)|\leq \mathcal{P}^n_t[|v_0|] (k) $. Finally, using the monotonicity preserving property of $ \mathcal{P}^n_t $ in Step 1 we conclude $ \mathcal{P}^n_t[|v_0|] (k) \leq \mathcal{P}^n_t[u_0] (k) $.
\end{proof}
As already indicated in Remark \ref{rem:LinearOp}, the linear semigroup $ \mathcal{P}^n_t $ is not well-defined for arbitrary functions $ u_0 $ as $ n\to\infty $. However, we are only interested in functions bounded by terms of the form $ |k|^p $ for $ p\geq2 $. Let us hence define
\begin{align*}
	u_{n,p}(k,t):=|k|^p\exp(-(\lambda_n(p)-p\norm{A})t),
\end{align*}
where $ \lambda_n(p) $ is given in \eqref{eq:EigenvalLinOp}. We can now prove the crucial comparison principle for the non-cutoff equation.
\begin{pro}\label{pro:ComparisionNonCutoff}
	Consider two weak solutions $ (f_t), \, (g_t)\subset\PrM_{p} $, $ p>2 $, to \eqref{eq:GenHomoengBE} with zero momentum. Let $ \varphi, \, \psi\in C([0,\infty);\F_p) $ be the corresponding Fourier transforms. Suppose that
	\begin{align*}
		|\varphi_0(k)-\psi_0(k)|\leq C_1|k|^p+C_2|k|^2, \quad \forall k\in\R^3.
	\end{align*}
	Then, we have for all $ t\geq0 $ and $ k\in\R^3 $
	\begin{align*}
		|\varphi_t(k)-\psi_t(k)|\leq C_1e^{-(\lambda(p)-p\norm{A})t}|k|^p+C_2e^{2\norm{A}t}|k|^2.
	\end{align*}
\end{pro}
\begin{proof}
	First of all, let us note that it is easy to prove that
	\begin{align*}
		u_{n,p}(k,t)\geq e^{-(S_n+A^\top k\cdot\nabla) t}|k|^p+\int_{0}^te^{-(S_n+A^\top k\cdot\nabla) (t-r)} \Lin_n(u_{n,p}(\cdot,r))dr.
	\end{align*}
	and hence, by a comparison principle for the linear equation, it follows
	\begin{align}\label{eq:ProofComparison}
		\mathcal{P}^n_t[|\cdot|^p](k) \leq u_{n,p}(t,k).
	\end{align}
	
	Now, we approximate $ \varphi,\, \psi $ by solutions $ \varphi^n,\, \psi^n\in C([0,\infty);\F_p) $ to equation \eqref{eq:FTgenBE} with cutoff kernel $ 0\leq b_n $ and initial datum $ \varphi_0 $ resp. $ \psi_0 $. We use an approximation as in the existence proof of \eqref{eq:GenHomoengBE}. For these we can apply Lemma \ref{lem:ComparisionCutoff} and \eqref{eq:ProofComparison}
	\begin{align}\label{eq:ComparisionCutoff}
		|\varphi_t^n(k)-\psi_t^n(k)|\leq  \mathcal{P}^n_t \big[C_1|\cdot|^p+C_2|\cdot|^2 \big]  (k) \leq C_1u_{n,p}(t,k)+C_2u_{n,2}(t,k).
	\end{align}
	Since weak convergence of measures implies pointwise convergence of their characteristic functions we can pass to the limit in \eqref{eq:ComparisionCutoff}. Recalling $ \lambda_n(p)\to\lambda(p) $ for $ p\geq 2 $, see Lemma \ref{lem:EigenfuncEigenvalLinOp}, concludes the proof.
\end{proof}
Finally, we prove the stability of the self-similar solutions or equivalently the steady states of \eqref{eq:StatGenBE}. Let us denote by $ \Psi=\F[f_{st}] $ the characteristic function of a steady state $ f_{st}\in\PrM_p $, $ 4\geq p>2 $ with second moments $ \bar{N} $, as in the first part of Theorem \ref{thm:MainTheorem}. Recall that $ \norm{A}\leq \varepsilon_0 $ is sufficiently small and $ 2\bar{\beta}=2\bar{\beta}(A) $ is the simple eigenvalue to the moment equations with eigenvector $ \bar{N}=\bar{N}(A) ,$ as in Lemma \ref{lem:SteadyStateMomenEq}.

To get an idea, recall the comparison principle in Proposition \ref{pro:ComparisionNonCutoff} and assume $ p\norm{A+\bar{\beta}}-\lambda(p)<0 $ for $ \norm{A+\bar{\beta}} $ sufficiently small. This would imply exponential convergence to the steady state, \emph{if} we knew that initially (or at some time) respective moments of the initial condition and $ f_{st} $ are equal, because then the quadratic term $ C_2|k|^2 $ would vanish. But this means in particular that the second moments are given by $ \bar{N} $ (and hence for all times). We cannot expect this, unless the solution equals the self-similar solutions. However, the overall idea is to show that the second moments converge to $ \alpha^2\bar{N} $, for some $ \alpha\geq0 $, exponentially in time. Applying again the comparison principle Proposition \ref{pro:ComparisionNonCutoff} yields the claim. A similar argument was also used in \cite{BobylevVelazq2020SelfsimilarAsymp} in the cutoff case.

\begin{proof}[Proof of Theorem \ref{thm:MainTheorem} (ii)] 
	Let $ (f_t)_t\subset \PrM_{p} $, $ 2<p\leq 4 $, be a solution to \eqref{eq:GenHomoengBE} with matrix $ A\in\R^{3\times3} $, $ \norm{A}\leq \varepsilon_{0} $. Here, $ \varepsilon_{0}>0 $ will be chosen to be small enough, in particular, such that part (i) of Theorem \ref{thm:MainTheorem} holds. Choosing $ (\tilde{f}_t)_t $ as stated in Theorem \ref{thm:MainTheorem} (ii) gives a solution to \eqref{eq:ApplicationGenBE} with matrix $ A+\bar{\beta} $ and zero momentum. Let us denote the characteristic functions of $ (\tilde{f}_t)_t $ by $ (\varphi_t)_t $ and the second moments by $ (M_t)_t $.
	
	\textit{Step 1.} We know that $ (M_t)_t $ satisfies the equation $ M_t' = (\mathcal{A}(\bar{b},A)-2\bar{\beta})M_t $. Furthermore, by Lemma \ref{lem:SteadyStateMomenEq} the nonzero eigenvalues of $ \mathcal{A}(\bar{b},A)-2\bar{\beta} $ have real part less than or equal to $ -\nu<0 $. The steady states are given by the span of $ \bar{N} $. Thus, there is $ C=C(M_0)\geq0 $ and $ \alpha=\alpha(M_0)\geq0 $ such that
	\begin{align*}
		\norm{M_t-\alpha^2\bar{N}}\leq Ce^{-\nu t}.
	\end{align*}
	
	\textit{Step 2.} Since $ \tilde{f}_t\in \PrM_{p} $, its characteristic function is in $ C^{2,p-2}_b $. W.l.o.g. we can assume $ p\leq 3 $. Now, we prove that $ \norm[C^{2,p-2}]{\varphi_t} $ is uniformly bounded in $ t\geq0 $. To this end we bound $ \sup_t \norm[p]{\tilde{f}_t} $. We use the Povzner estimates as in Lemma \ref{lem:UnifMomentBound} to get
	\begin{align*}
		\dfrac{d}{dt}\norm[p]{\tilde{f}_t}&\leq p\norm{A+\bar{\beta}}\norm[p]{\tilde{f}_r} +C'\Lambda \norm[2]{\tilde{f}_r}^2 - c'\Lambda\norm[p]{\tilde{f}_r}
		\\
		&\leq \left[ p\varepsilon_0(1+c_0)-c'\Lambda\right] \norm[p]{\tilde{f}_r} + C'\Lambda K^2.
	\end{align*}
	Here, $ \sup_t \norm{M_t}\leq K $ by Step 1 and $ c_0>0 $ is from Lemma \ref{lem:SteadyStateMomenEq}. For $ \varepsilon_0=\varepsilon_0(p,b) $ sufficiently small we have $ \delta:=c'\Lambda-p\varepsilon_0(1+c_0)>0 $ and hence, from a Gronwall type argument, the desired uniform bound
	\begin{align*}
		\norm[p]{\tilde{f}_t} \leq \norm[p]{\tilde{f}_0}e^{-\delta t} + \dfrac{C'\Lambda K^2}{\delta} =:C_*(K,\varepsilon_0).
	\end{align*} 
	
	\textit{Step 3.} We recall that $ \Psi=\F[f_{st}] $, where $ f_{st} $ is the steady state with second moments $ \bar{N} $. Note that $ \Psi(\alpha\cdot) $ is the characteristic function of the steady state $ \alpha^{-3}f_{st}(v/\alpha) $ with second moments $ \alpha^2\bar{N} $. We estimate the characteristic functions
	\begin{align*}
		|\varphi_t(k)-\Psi(\alpha k)| &\leq \left| \varphi_t(k)-1+\dfrac{1}{2}M_t:k\otimes k \right| 
		\\
		&+\dfrac{1}{2}\norm{M_t-\alpha^2\bar{N}}|k|^2+\left|1-\dfrac{1}{2}\alpha^2 \bar{N}:k\otimes k - \Psi(\alpha k) \right|.
	\end{align*}
	For the first term we use a Taylor expansion, in conjunction with the fact that $ D^2\varphi_t $ is $ (p-2) $-Hölder continuous with $ \norm[C^{p-2}]{D^2\varphi_t}\leq C_* $. We obtain the estimate
	\begin{align*}
		\left| \varphi_t(k)-1+\dfrac{1}{2}M_t:k\otimes k \right| \leq C_*|k|^p.
	\end{align*}
	The last term is treated similarly due to $ \Psi\in\F_p $ and for the second term we apply Step 1. All together this yields
	\begin{align*}
		|\varphi_t(k)-\Psi(\alpha k)|\leq C|k|^{p}+ Ce^{-\nu t}|k|^2.
	\end{align*}
	Now, we apply the comparison principle in Proposition \ref{pro:ComparisionNonCutoff} starting at time $ T $ to obtain
	\begin{align*}
		|\varphi_{T+t}(k)-\Psi(\alpha k)|\leq Ce^{-(\lambda(p)-p\norm{\bar{\beta}+A})t}|k|^p+Ce^{-\nu T+2\norm{\bar{\beta}+A}t}|k|^2.
	\end{align*}
	We recall that for all previous arguments we had to choose $ \norm{A}\leq \varepsilon_0  $ sufficiently small. We now further assume that $ \varepsilon_0>0 $ is small enough to ensure
	\begin{align*}
		\norm{\bar{\beta}+A} \leq (1+c_0)\norm{A} \leq \min\left( \dfrac{\lambda(p)}{2p},\dfrac{\nu}{4} \right).
	\end{align*}
	Thus, we get for $ t=T $ and $ \theta'=\min(\frac{\lambda(p)}{4},\frac{\nu}{4}) $
	\begin{align}\label{eq:ExpConvergenceEstimate}
		|\varphi_{2T}(k)-\Psi(\alpha k)|\leq Ce^{-2\theta' T}\left( |k|^p+|k|^2\right),
	\end{align}
	where $ C=C(\varphi_0,p) $. Now, we apply the following inequality valid for all $ \varphi,\psi\in\F_p $
	\begin{align}\label{eq:Interpolation}
		d_2(\varphi,\psi)\leq c_p (\gamma+\gamma ^{2/p}), \quad \gamma:= \sup_k \dfrac{|\varphi-\psi|(k)}{|k|^2+|k|^{p}}.
	\end{align}
	This can be proved by splitting the supremum in $ d_2(\varphi,\psi) $ into $ |k|\leq R $ and $ |k|\geq R $ and minimizing over $ R $. Combining both \eqref{eq:ExpConvergenceEstimate} and \eqref{eq:Interpolation} yields 
	\begin{align*}
		d_2(\varphi_t,\Psi(\alpha\cdot)) \leq Ce^{-\theta t},
	\end{align*}
	for some $ \theta>0 $.
\end{proof}

\subsection{Finiteness of higher moments}
To prove part (iii) of Theorem \ref{thm:MainTheorem}, we need an extension of Lemma \ref{lem:UnifMomentBound}.
\begin{lem}\label{lem:UnifMomentBoundM}
	Let $ M\in\N $, $ M\geq3 $ and $ p\geq M $. Consider the unique solution $ (f_t)\in C([0,\infty); \PrM_{p}) $ to \eqref{eq:GenHomoengBE} with $ A $ replaced by $ A+\bar{\beta} $. Let $ \norm{A}\leq \varepsilon_0 $ and $ \varepsilon_0>0 $ from Lemma \ref{lem:SteadyStateMomenEq}. Assume that the initial condition $ f_0\in\PrM_{p} $ satisfies
	\begin{align}
		\int_{\R^3}vf_0(dv)=0, \quad \int_{\R^3}v_iv_jf_0(dv)=K\bar{N}_{ij}.
	\end{align}
	Then, there is $ \varepsilon_M\leq \varepsilon_0 $ and $ C_*=C_*(K,M) $ such that: if $ \norm{A}\leq\varepsilon_M  $ then
	\begin{align*}
		\norm[M]{f_0}\leq C_* \implies \norm[M]{f_t}\leq C_*
	\end{align*}
	for all $ t\geq0 $.
\end{lem}
\begin{proof} 
	This can be proved by induction over $ M $ by applying repeatedly the Povzner estimate, Lemma \ref{lem:Povzner}. The case $ M=3,4 $ is covered by Lemma \ref{lem:UnifMomentBound} and at each step one has to choose $ \varepsilon_M\leq \varepsilon_{M-1} $ and $ \norm{A}\leq \varepsilon_M $ to absorb the drift term.
\end{proof}
\begin{proof}[Proof of Theorem \ref{thm:MainTheorem}. (iii).]
	We argue as in the existence proof in Subsection \ref{subsec:Existence}. However, now we include the uniform bound $ \norm[M]{f}\leq C_*(M,K) $ in the definition of the sets $ \mathscr{U} $. The so constructed stationary solutions coincide with the ones in part (i) of Theorem \ref{thm:MainTheorem} by uniqueness.
\end{proof}

\section{Application to simple and planar shear}\label{sec:Application}
In this section, we will discuss the long-time behavior of homoenergetic solutions in the case of simple and planar shear. Recall that homoenergetic flows have the form $ g(t,x,v)= f(t,v-L(t)x) $ and $ f=f(t,v) $ satisfies
\begin{align}\label{eq:ApplicationHomoenFlow}
	\partial_t f -L(t)v\cdot \nabla f = Q(f,f)
\end{align}
with the matrix $ L(t)=(I+tL_0)^{-1}L_0 $. Under the assumption $ \det(I+tL_0)>0 $ for all $ t\geq0 $, one can study the form of $ L(t) $ as $ t\to \infty $ (see \cite[Sect. 3]{JamesVelazq2019SelfsimilarProfilesHomoenerg}). We consider the case of simple shear resp. planar shear ($ K\neq 0 $)
\begin{align*}
	L(t)=\left( \begin{array}{ccc}
		0 & K & 0 \\ 
		0 & 0 & 0 \\ 
		0 & 0 & 0
	\end{array} \right) \quad \text{resp.}\quad 
	L(t)=\dfrac{1}{t}\left( \begin{array}{ccc}
		0 & 0 & 0 \\ 
		0 & 0 & K \\ 
		0 & 0 & 1
	\end{array} \right) +\mathcal{O}\left( \dfrac{1}{t^2}\right) \quad (t\to \infty).
\end{align*}
In the first case, \eqref{eq:ApplicationHomoenFlow} preserves mass, since $ \trace L=0 $, and our study applies for $ K $ sufficiently small. Alternatively, one can assume a largeness condition on the kernel $ b $, see the assumption below.

Let us now turn to planar shear and write $ L(t)=A/(1+t)+\tilde{A}(t) $ with $ \trace A=1 $, $ \norm{\tilde{A}(t)}\leq\mathcal{O}\left( 1/(1+t)^2\right) $. First, let us introduce the time-change $ \log(1+t)=\tau $ and set $ f(t,v)=F(\tau,v)/(t+1) $ yielding the equation (after multiplying with $ (1+t)^2 $)
\begin{align}\label{eq:ApplicationModBE}
	\partial_\tau F - \div((A+B(\tau))v\cdot F) + \trace B(\tau) F = Q(F,F)
\end{align}
where $ B(\tau)=(1+t)\tilde{A}(t)=\mathcal{O}\left( 1/(1+t)\right)=\mathcal{O}(e^{-\tau}) $.

We will apply our results, Theorem \ref{thm:MainTheorem} $ (ii) $, to \eqref{eq:ApplicationModBE} yielding a self-similar asymptotic. Let us note that for a cutoff kernel the well-posedness theory is clearly stable enough to deal with the perturbation $  B(\tau)=\mathcal{O}( e^{-\tau}) $. We will show that this is also the case here for Maxwell molecules with a singular kernel. 

The well-posedness theory of \eqref{eq:ApplicationModBE} does not change at all comparing with \eqref{eq:GenHomoengBE} and so we omit further details about existence, uniqueness and regularity.  Below we will sketch the prove of convergence of $ F $ to a self-similar profile, which is analogous to Theorem \ref{thm:MainTheorem}. More precisely, we will show the following result (note that we write $ t $ instead of $ \tau $ in the theorem again). 

\begin{thm}\label{thm:Application}
	Consider \eqref{eq:ApplicationModBE} with $ A\in\R^{3\times 3} $ and $ B_t \in C([0,\infty);\R^{3\times 3}) $ such that $ \norm{B_t}=\mathcal{O}(e^{-t}) $. Let $ (F_t)_t\subset\PrM_{p} $, $ 2<p $, be a weak solution to \eqref{eq:ApplicationModBE} with $ F_0\in\PrM_{p} $ and first moments
	\begin{align*}
		\int_{\R^3}vF_0(dv)=U.
	\end{align*}
	We define $ m_t\in\R $, $ E_t\in\R^{3\times 3} $ as follows
	\begin{align}\label{eq:ApplicationMass}
		\begin{split}
			m_t&=\int F_t(v)dv=\exp\left( -\int_{0}^t\trace B_s \, ds \right), \quad \lim_{t\to \infty}m_t = m_\infty, 
			\\
			E_t'&=(A+B_t)E_t, \quad E_0=I.
		\end{split}
	\end{align}
	There is a constant $ \varepsilon_{0}=\varepsilon_0(m_\infty b,p)>0 $ such that for $ \norm{A}\leq \varepsilon_0 $, the following holds. Defining
	\begin{align}\label{eq:ApplicationThmTransformations}
		\tilde{F}_t:= \dfrac{e^{3\bar{\beta} t}}{m_t} F_t\left(e^{\bar{\beta}t} v +E_t U \right), \quad \tilde{f}_{st}(v)=f_{st}(v \alpha_\infty^{-1})\alpha_\infty^{-3}
	\end{align}
	for a constant $ \alpha_\infty= \alpha_\infty(F_0) $ we have the bound ($ \lambda>0 $)
	\begin{align*}
		d_2(\tilde{F}_t,\tilde{f}_{st}) \leq Ce^{-\lambda t}.
	\end{align*}
	Here, $ f_{st}\in\PrM_{p} $ is the solution to
	\begin{align*}
		\div ((A+\bar{\beta})v \cdot f_{st}) + m_\infty Q(f_{st},f_{st}) =0
	\end{align*} 
	where $ \bar{\beta}=\bar{\beta}(A) $ and $ f_{st} $ has second moments $ \bar{N}=\bar{N}(A) $ as in Theorem \ref{thm:MainTheorem} but with collision kernel $ m_\infty b $.
\end{thm}
\begin{rem}
	Let us stress that $ \varepsilon_{0}=\varepsilon_{0}(m_\infty b) $ does not depend on $ F $, since $ m_\infty $ is given by \eqref{eq:ApplicationMass}. The constant will have similar dependencies as in Theorem \ref{thm:MainTheorem}. Furthermore, the functions \eqref{eq:ApplicationMass} are needed to deal with the time-dependent mass and momentum.
\end{rem}

With this let us now go back to solutions $ (f_t)_t $ of equation \eqref{eq:ApplicationHomoenFlow} with $ L(t)=A/(1+t)+\tilde{A}(t) $. To apply the previous result, we need $ \norm{A}\leq \varepsilon_0 $. This might not be true for $ A $ coming from the matrix $ L(t) $ above. However, one can instead assume a largeness condition on the kernel $ b $. To see this, let us rescale time $ \tau\mapsto\tau M $ yielding
\begin{align}\label{eq:ApplicationRescaling}
	\partial_\tau F - \dfrac{1}{M}\div((A+B(\tau))v\cdot F) + \dfrac{1}{M}\trace B(\tau) F = \dfrac{1}{M}Q(F,F).
\end{align}
In particular, the collision kernel is given by $ b/M $. We can hence consider the following assumption.

\paragraph{Assumption.} Assume that the kernel $ b $ is chosen such that
\begin{align*}
	\norm{A/M} \leq \varepsilon_0(m_\infty b/M)
\end{align*} 
is satisfied for some (maybe sufficiently large but fixed) $ M>0 $.

A similar condition was also used in \cite[Section 5.2]{JamesVelazq2019SelfsimilarProfilesHomoenerg}.  Under this assumption, we can apply Theorem \ref{thm:Application}. Let us state now the asymptotics in terms of $ (f_t)_t $ solving \eqref{eq:ApplicationHomoenFlow}. For this we undo the above transformations, i.e. the logarithmic time change $ f(t,v)=F(\tau,v)/(t+1) $, the renormalization in \eqref{eq:ApplicationThmTransformations} and the scaling of time $ \tau\mapsto\tau M $. We obtain for the solution $ f $ of \eqref{eq:ApplicationGenBE}
\begin{align}\label{eq:SelfSimilarConvergenceResult}
	\dfrac{e^{t/M}e^{3\bar{\beta} t}}{m_t} f \left(e^{t/M}-1, e^{\bar{\beta}t} v +E_t U \right)  \rightarrow f_{st}(v \alpha^{-1})\alpha^{-3} \qquad \text{as}\; t\to \infty.
\end{align}
Here, $ U\in\R^3 $ is the mean of the initial condition $ f_0\in \PrM_{p} $ and $ \alpha=\alpha(f_0) $ is as in Theorem \ref{thm:Application}. For $ B_\tau:=e^{\tau}\tilde{A}(e^{\tau}-1) $ we defined
\begin{align*}
	m_t=\exp\left( -\dfrac{1}{M}\int_{0}^t\trace B_s \, ds \right), \quad E_t'=\dfrac{1}{M}(A+B_t)E_t, \quad E_0=I.
\end{align*}

The convergence above appears with an order $ \mathcal{O}( e^{-\lambda \tau} ) $, $ \tau =e^{t/M}-1 $ in \eqref{eq:SelfSimilarConvergenceResult}, w.r.t the metric $ d_2 $. Thus, on the normal time scale the convergence has order $ \mathcal{O}( t^{-\lambda M} ) $. Although this gets better as $ M $ increases, our assumption on the collision kernel $ b $ would become more restrictive.

Finally, let us give the main arguments for the proof of Theorem \ref{thm:Application}, following our previous analysis in Section \ref{sec:ExUniStabSelfSimilar}.

\begin{proof}[Proof of Theorem \ref{thm:Application}]
	\textit{Preparation.}  First, note that equation \eqref{eq:ApplicationModBE} does not preserve mass. So we need to rescal the solution. Furthermore, we can set the momentum to zero. More precisely, with \eqref{eq:ApplicationMass} we define $ G_t(v)=F_t(v+E_t U)/m_t $ and obtain
	\begin{align*}
		\partial_t G -\div((A+B_t)v\cdot G) = m_tQ(G,G), \quad \int G_t(dv)=1, \quad \int vG_t(dv)=0.
	\end{align*}
	The assumption $ \norm{B_t}=\mathcal{O}(e^{-t}) $ implies $ m_t\to m_\infty>0 $ and $ |m_{T+t}-m_T|\leq Ce^{-T} $. We introduce the self-similar variables $ G_t(v)=f_t(ve^{-\bar{\beta}t})e^{-3\bar{\beta}t} $ and get
	\begin{align}\label{eq:ApplicationGenBE}
		\partial_t f -\div((A+\bar{\beta}+B_t)v\cdot f) = m_tQ(f,f).
	\end{align}
	where $ \bar{\beta}=\bar{\beta}(A,m_\infty\bar{b}) $ is as in Theorem \ref{thm:MainTheorem} or Lemma \ref{lem:SteadyStateMomenEq} when considering the collision kernel $ m_\infty\bar{b} $.
	
	Now, the plan is as follows. First, we study the longtime behavior of the second moments $ M_t $ of $ f_t $ in Step 1. Then, in Step 2, we want to compare \eqref{eq:ApplicationGenBE} to solutions $ g^{(T)} $ of
	\begin{align}\label{eq:ApplicationTimeGenBE}
		\partial_t g^{(T)} =\div\left((A+\bar{\beta})v \cdot g^{(T)} \right)+m_\infty Q\left( g^{(T)},g^{(T)}\right) , \quad g^{(T)}_0=f_T.
	\end{align}
	This equation has $ f_{st} $ as stationary solution. In Step 3, we apply Theorem \ref{thm:MainTheorem} to $ g^{(T)} $ to obtain $ g^{(T)}\to f_{st}(\alpha_T^{-1}\cdot)\alpha_T^{-3} $. Finally, we conclude with all this $ f_t\to f_{st}(\alpha_\infty^{-1}\cdot)\alpha_\infty^{-3} $. Here, $ \alpha_T,\, \alpha_\infty $ are constants, which precise values will be apparent below.
	
	\textit{Step 1.} Let $ M_t $ be the second moments of $ f_t $, which satisfy (see also Lemma~\ref{lem:MomentEq})
	\begin{align*}
		\dfrac{dM_t}{dt} = (2\bar{\beta}+\mathcal{A}(m_t \bar{b},A)) M_t + \mathcal{B}_t M_t, 
	\end{align*}
	where $ \mathcal{A}(m_t \bar{b},A+\bar{\beta}) $, $ \mathcal{B}_t  $ are linear operators $ \R^{3\times 3}_{sym}\rightarrow \R^{3\times 3}_{sym} $ and $ \norm{\mathcal{B}_t}\leq Ce^{-t} $. The first operator corresponds to the drift term with matrix $ A+\bar{\beta} $ and the collision operator. The second operator captures the drift term with $ B_t $. Due to the simple dependence of $ \mathcal{A} $ w.r.t. $ m_\infty\bar{b} $ (see Lemma \ref{lem:MomentEq}) we can write
	\begin{align*}
		\dfrac{dM_t}{dt} = \mathcal{A}(m_\infty \bar{b},A+\bar{\beta}) M_t + \mathcal{R}_t M_t,
	\end{align*}
	where we put $ m_\infty $ in the first term. Since $ |m_t-m_\infty|\leq Ce^{-t} $ we still have $ \norm{\mathcal{R}_t}\leq Ce^{-t} $. The results of Lemma \ref{lem:SteadyStateMomenEq} hold for the semigroup $ e^{\mathcal{A} t} $ generated by $\mathcal{A}:= \mathcal{A}(m_\infty \bar{b},A+\bar{\beta}) $. Using Duhamel's formula one can prove that
	\begin{align*}
		e^{\mathcal{A} t}M_T \to \alpha_T^2\bar{N}, \quad M_t\to \alpha_\infty^2 \bar{N}
	\end{align*}
	as $ t\to \infty $ for all $ T\geq0 $ with a convergence of order $ Ce^{-\nu t} $. Furthermore, $ |\alpha_\infty^2-\alpha_T^2|\leq Ce^{-T} $ where the constants $ C>0 $ are always independent of $ T $.
	
	\textit{Step 2.} Now, we compare $ f $ with $ g^{(T)} $ satisfying \eqref{eq:ApplicationTimeGenBE}. We need for all $ t, \,T\geq0 $ the following estimate
	\begin{align}\label{eq:ApplicationStability}
		d_2\left( f_{t+T},g^{(T)}_t \right) \leq Ct \,e^{-T+2\norm{A+\bar{\beta}}t}
	\end{align} 
	with $ C $ independent of $ t,T $. This works in the same way as the uniqueness proof of Proposition \ref{pro:IntExistUniqReg} in Subsection \ref{subsec:Uniquenss}. The difference here is the coefficient $ m_t $ in front of the collision operator, as well as the term due to $ B_t $ in \eqref{eq:ApplicationGenBE}. Both of them lead to a term of order $ e^{-T} $ and after the Gronwall argument one obtains the inequality
	\begin{align*}
		e^{m_\infty S_nt}d_2(\varphi_{t+T},\psi_t)\leq 
		\left( r_n+ Ce^{-T}\right) \int_0^te^{m_\infty S_nr} e^{\left[ 2\norm{A+\bar{\beta}}+m_\infty S_n\right](t-r)}dr.
	\end{align*}
	After canceling the exponential terms containing $ m_\infty S_n $, we let $ n\to \infty $ and estimate the integral on the right-hand side to get \eqref{eq:ApplicationStability}.
	
	\textit{Step 3.} Now, we apply Theorem \ref{thm:MainTheorem} to the solutions $ g^{(T)} $ to \eqref{eq:ApplicationTimeGenBE}. For this, let $ f_{st} $ be the stationary solution to \eqref{eq:ApplicationTimeGenBE} with second moments $ \bar{N} $ and $ \Psi=\F[f_{st}] $. We get in Fourier space $ \psi_t^{(T)}=\F[g^{(T)}] $, with $ \alpha_T $ as in Step 1,
	\begin{align*}
		d_2\left( \psi^{(T)}_t, \Psi(\alpha_T\,\cdot)\right) \leq Ce^{-\theta t}.
	\end{align*}
	The only problem now is that the constant $ C $ might depend on the initial condition $ f_T $ and thus on $ T $. If we trace back the dependence of this constant in the proof of Theorem \ref{thm:MainTheorem} (ii), then two constants $ C_1,C_2 $ contribute. The first one satisfies
	\begin{align*}
		\norm{e^{\mathcal{A}_{\bar{\beta}} t}M_T-\alpha_T^2 \bar{N}}\leq C_1e^{-\nu t}
	\end{align*}
	and depends only on $ M_T $, which is uniformly bounded. The second constant is a uniform bound on the moments of order $ 4\geq p>2 $, see Step 2 in the proof of Theorem \ref{thm:MainTheorem} (ii). Looking at the arguments there, we see that it suffices to show $ \sup_{t}\norm[p]{f_t} <\infty $ in order to obtain $ \sup_{t,T}\norm[p]{g_t^{(T)}}<\infty $. This can be proved again by an application of the Povzner estimate to the equation \eqref{eq:ApplicationGenBE}. The difference here is an additional term due to $ B_t $. Since this is integrable in time one can choose $ \varepsilon_{0}>0 $ small enough in exactly the same way.
	
	\textit{Conclusion.} Let us combine all our estimates in Fourier space $ \varphi_t=\F[f_t] $, $ \psi_t^{(T)}=\F[g^{(T)}] $
	\begin{align*}
		d_2\left( \varphi_{t+T},\Psi(\alpha_\infty \, \cdot) \right)&\leq d_2\left(\varphi_{t+T},\psi^{(T)}_t \right) + d_2\left(\psi^{(T)}_t ,\Psi(\alpha_T \,\cdot) \right) + d_2\left( \Psi(\alpha_T \,\cdot), \Psi(\alpha_\infty \,\cdot) \right)
		\\
		&\leq Ct\, e^{-T+2\norm{A+\bar{\beta}}t}+Ce^{-\theta t}+ Ce^{-T}.
	\end{align*}
	The first two estimates are clear by Step 2 respectively Step 3. The last one follows from a Taylor expansion and $ |\alpha_\infty^2-\alpha_T^2|\leq Ce^{-T} $, as we know from Step~1. Let us now choose $ t=T $ and ensure $ 2\norm{A+\bar{\beta}}\leq 2(1+c_0)\norm{A}\leq 1/2 $, which is just a smallness assumption on $ \norm{A} $ that we used in a similar form for Theorem \ref{thm:MainTheorem}. This concludes the proof.
\end{proof}

\bibliographystyle{abbrv}
\bibliography{Selfsimilar_solutions_homoenergBE_noncutoff}

\end{document}